\documentclass[11pt,reqno,a4paper]{amsart}
\usepackage{mathrsfs}
\usepackage{amsgen}
\usepackage{amscd}
\usepackage{booktabs}
\usepackage{amsmath}
\usepackage{latexsym}
\usepackage{amsfonts}
\usepackage{amssymb}
\usepackage{amsthm}
\usepackage{graphicx,epsfig,subfigure}
\usepackage{color}
\usepackage{amsthm,amssymb}
\usepackage{graphicx}
\usepackage{sidecap}
\usepackage{enumerate}
\usepackage{amssymb,amsmath,graphicx,amsfonts,euscript}
\usepackage{color}
\usepackage{mathrsfs}
\usepackage{empheq}
\usepackage{cases}
\usepackage{cite}
\usepackage{ulem}
\usepackage{cancel}
\usepackage[margin=1.0in]{geometry}
\usepackage{amssymb,amsmath,graphicx,amsfonts,euscript,mathrsfs}
\usepackage{hyperref}

\usepackage{titlesec}
\titleformat{\section}{\vskip5pt\large\bfseries}{\thesection.}{0.5em}{\centering}
\titleformat{\subsection}{\vskip5pt\normalsize\bfseries}{\thesubsection.}{0.5em}{}
\allowdisplaybreaks

\newtheorem{theorem}{Theorem}[section]
\newtheorem{lemma}[theorem]{Lemma}

\newtheorem{remark}[theorem]{Remark}
\theoremstyle{definition}

\newtheorem{corollary}{Corollary}

\def\R{\mathbb{R}}

\def\T{\mathbb{T}}

\def\U{\mathcal{U}}
\def\F{\mathcal{F}}

\def\d{\mathrm{d}}

\numberwithin{equation}{section}

\begin{document}

\title[]{Computing rough solutions of the KdV equation below ${\bf L^2}$}

\author[]{Jiachuan Cao}
\address{Jiachuan Cao: Department of Applied Mathematics, The Hong Kong Polytechnic University,
Hung Hom, Hong Kong}
\email{jiachuan.cao@polyu.edu.hk}

\author[]{Buyang Li}
\address{Buyang Li: Department of Applied Mathematics, The Hong Kong Polytechnic University,
Hung Hom, Hong Kong}
\email{buyang.li@polyu.edu.hk}

\author[]{Yifei Wu}
\address{Yifei Wu: School of Mathematical Sciences, Nanjing Normal University, Nanjing, China}
\email{yerfmath@gmail.com}

\author[]{Fangyan Yao}
\address{Fangyan Yao: Department of Applied Mathematics, The Hong Kong Polytechnic University,
Hung Hom, Hong Kong}
\email{fangyan.yao@polyu.edu.hk}




\maketitle
\vspace{-20pt}

\begin{abstract}\noindent
{\small 
We establish a novel numerical and analytical framework for solving the Korteweg--de Vries (KdV) equation in the negative Sobolev spaces, where classical numerical methods fail due to their reliance on high regularity and inability to control nonlinear interactions at low regularities. Numerical analysis is established by combining a continuous reformulation of the numerical scheme, the Bourgain-space estimates for the continuous reformulation, and a rescaling strategy that reduces the reformulated problem to a small initial value problem, which allow us to bridge a critical gap between numerical analysis and theoretical well-posedness by designing the first numerical method capable of solving the KdV equation in the negative Sobolev spaces. The numerical scheme is proved to have nearly optimal-order convergence with respect to the spatial degrees of freedom in the $H^{-\frac{1}{2}}$ norm for initial data in $H^s$, with $-\frac{1}{2} < s \leq 0$, a result unattainable by existing numerical methods.}
\end{abstract}


\setlength\abovedisplayskip{4pt}
\setlength\belowdisplayskip{4pt}

\section{Introduction}

The Korteweg-de Vries (KdV) equation is a fundamental mod- el in the study of dispersive partial differential equations, with applications ranging from shallow water waves to acoustic wave propagation in plasmas and lattices. 
Over the past decades, significant research has been devoted to both the mathematical theory and numerical solutions of the KdV equation. 

From a theoretical perspective, research on the well-posedness of the KdV equation has witnessed a progressive relaxation of the regularity requirements for the initial data. Bourgain \cite{Bourgain} first established the well-posedness of \eqref{eq:KdV_equation} for initial data in $H^s$ with $s \geq 0$. This result was later extended by Kenig, Ponce, and Vega to the range $s > -\frac{1}{2}$ \cite{Kenig} and refined further at the endpoint by Colliander et al. \cite{Colliander} using Bourgain spaces. 
The range $s\ge -\frac12$ is sharp if one requires the data to solution map to be uniformly continuous on bounded sets.
Recently, the global well-posedness of the KdV equation in the $H^{-1}$ space in the sense that the solution map is only continuous, was achieved through a novel methodology for integrable PDEs \cite{KV-2019}.

On the numerical side, a wide range of methods have been developed to approximate solutions of the KdV equation, including finite difference methods, splitting methods, discontinuous Galerkin methods, and classical exponential integrators \cite{splittingJCP,Spectralkdv1,Fourierkdv,Spectralkdv2,splitting0,splitting1,splitting2,DG2,Ostermann-Su-2020}. However, classical numerical methods typically require the solution to be sufficiently smooth due to their reliance on the boundedness of the time derivative of $u(t)$. For the KdV equation 
on the torus  $\T=[0,1]$, 
\begin{equation}\label{eq:KdV_equation}
	\left\{
	\begin{array}{l}
		{\displaystyle \partial_t u(t,x) + \partial_x^3 u(t,x) = \frac{1}{2} \partial_x \big(u(t,x)^2\big), \quad x \in \mathbb{T}, \; t \in (0,T],} \\[2mm]
		{\displaystyle u(0,x) = \phi(x), \quad x \in \mathbb{T},}
	\end{array}
	\right.
\end{equation}
classical time discretization methods typically require the solution to belong to 
$H^3$ in order to achieve first-order convergence in the $L^2$ norm. However, this regularity requirement is too restrictive for computing rough solutions of the KdV equation, particularly when the solution lacks differentiability or even continuity.

In response to these challenges, the development of low-regularity integrators (LRIs) has garnered significant attention. These advanced numerical techniques are designed to relax the regularity constraints on the solution while maintaining stability and accuracy. The first LRI for the KdV equation was introduced by Hofmanová and Schratz \cite{Hofmanova-Schratz-2017}, based on the approximation:
\begin{align}
	u({\sigma_n} + \tau) \approx &e^{-\tau \partial_x^3} u(\sigma_n) + \frac{1}{2} \int_{\sigma_n}^{\sigma_{n+1}} e^{(t - \sigma_{n+1}) \partial_x^3} \partial_x \left( e^{(\sigma_n - t) \partial_x^3} u(\sigma_n) \right)^2 \, dt, \quad  \sigma_n=n\tau.
\end{align}
They proved that this method achieves first-order convergence in $H^1$ for initial data in $H^3$. Later, Wu and Zhao extended this approach to construct LRIs with second-order convergence \cite{Wu-2}. Specifically, their method achieves second-order convergence in $H^\gamma$ when the initial data lies in $H^{\gamma+4}$, where $\gamma \geq 0$. 
Wu and Zhao \cite{Wu-3} also proposed two embedded low-regularity integrators tailored for different regularity settings. The first achieves first-order convergence in $H^\gamma$ for initial data in $H^{\gamma+1}$ for $\gamma > \frac{1}{2}$, while the second achieves second-order convergence in $H^\gamma$ for initial data in $H^{\gamma+3}$ regularity for $\gamma \geq 0$. These advancements illustrate the potential of LRIs in addressing the limitations of classical numerical methods.
Beyond the KdV equation, LRIs have been successfully applied to other nonlinear dispersive equations, such as the nonlinear Schrödinger equation, nonlinear wave equations, and modified KdV equations. For a comprehensive overview of these developments, see \cite{BBS-2024-FoCM,BLW2022,BS,CLL2023,CLLY2024,LW2021,Li-1,LMS,LSZ,NWZ-2022,Ostermann-Schratz-FoCM,Ostermann-Su-2020,RS21,RS-2024-FoCM,Wu-Yao-MCOM} and the references therein.

In \cite{Hofmanova-Schratz-2017, Wu-2, Wu-3}, the minimal regularity requirement for the initial data remains $\frac{3}{2}+$. 
This limitation arises from the treatment of the Burgers-type nonlinear term in the KdV equation, which involves a first-order derivative.
To address cases with lower regularity, Rousset and Schratz \cite{RS-PAM-2022} introduced three filtered time discretization schemes. 
For $u \in C([0,T]; H^s)$ for $s \in (0,3]$, these methods achieve $\frac{s}{3}$-order convergence in the $L^2$ space, utilizing the framework of discrete Bourgain space analysis. This analytical approach was originally established in \cite{ORS-JEMS} for the nonlinear Schrödinger equation. For further applications of discrete Bourgain spaces in numerical methods for other equations, readers are referred to \cite{JLOS-2024-1, JORS-2024-1, JORS-2024-2, ORS-MC}.
More recently, for initial data $u(0) \in H^s(\mathbb{T})$ with $s \in (0,1]$, Li and Wu \cite{Li-1} developed numerical methods with enhanced convergence rates. By employing advanced harmonic analysis techniques, they constructed a scheme that achieves $s$-order convergence in $L^2$.

Despite significant progress, all previously known methods still require strictly positive Sobolev regularity. In contrast, theoretical studies have established that the KdV equation is well-posed in the negative regularity regime. However, classical numerical methods fail in this setting due to their reliance on high regularity and their inability to effectively control nonlinear interactions at extremely low regularities.




In this article, we tackle this challenging problem by introducing a novel numerical analytic framework based on Bourgain spaces and a continuous reformulation strategy inspired by~\cite{Bourgain, Colliander, Kenig, Li-1}. Our approach achieves a convergence order for rough solutions that surpasses all previously known methods, thereby significantly extending the applicability of numerical simulations to physically realistic scenarios, such as those involving near-white noise or Gibbs measure initial data. The main contributions of our work can be summarized clearly as follows:
\begin{itemize}
	\item \textit{First numerical method below $L^2$ regularity:} We present the first numerical integrator that is rigorously proven to be convergent for the KdV equation with initial data in the negative regularity regime $ H^s(\mathbb{T}) $, for $ s \in (-\tfrac{1}{2}, 0] $, achieving a convergence rate that surpasses all previously known numerical methods (see Theorem~\ref{thm:main}). In particular, for $ s = 0 $ (i.e., $ L^2 $ initial data), the proposed method converges in $ H^{-1/2} $ with an order close to $ \tfrac{1}{4} $ — a rate that was previously unattainable.
	
	\item \textit{Optimal convergence order for extremely rough solutions:}  
	Our fully discrete numerical scheme achieves optimal spatial convergence rates for initial data in $H^s(\mathbb{T})$, $s\in(-\frac{1}{2},0]$. Specifically, when applying the Fourier spectral method for spatial discretization under the CFL-type condition $\tau=O(N^{-2+2s-})$, we rigorously establish an optimal convergence rate of $N^{-\frac{1}{2}-s}$ in the $H^{-\frac{1}{2}}(\mathbb{T})$-norm (see Corollary~\ref{rem:fully_discrete}).
	
	\item \textit{Novel analytical framework based on Bourgain spaces and  continuous  reformulation:}  
	To overcome the challenge arising from the failure of standard estimates in controlling the nonlinear term at very low regularities, we introduce a novel analytical framework. Our approach combines Bourgain-space techniques with a carefully designed  continuous  reformulation of the original time-discrete numerical scheme, inspired by the methodology in~\cite{Li-1}. This innovative reformulation allows us, for the first time, to directly utilize powerful Bourgain-space estimates developed for dispersive equations (see e.g.~\cite{Bourgain,Colliander,Kenig}). Consequently, we rigorously control the nonlinear interactions at extremely low regularities, effectively resolving longstanding analytical barriers in the numerical analysis of dispersive PDEs.
	
	\item \textit{Rescaling and continuation techniques:}  
	To overcome the critical lack of additional smallness arising in the Bourgain-space estimates at the endpoint regularity $X_{s,\frac{1}{2}}$ (see e.g.~Lemma 2.11 of~\cite{Tao}), we introduce a rescaling argument inspired by Colliander et al.~\cite{Colliander}. Specifically, by simultaneously rescaling both the original PDE and the numerical scheme, we effectively transform the continuous-in-time reformulation of our discrete scheme into a small-initial-data problem. This rescaling strategy yields two crucial advantages: firstly, it allows us to prove the uniform boundedness of the reformulated numerical solutions in Bourgain spaces independently of the time step $\tau$; secondly, it resolves the critical absence of additional smallness inherent to endpoint Bourgain-space cases, enabling rigorous control of nonlinear integral terms even in large-data regimes (see, for instance, estimates \eqref{eq:e_estimate}--\eqref{eq:condition_lambda} in Section~\ref{sec:proof_of_main_theorem}). These techniques effectively address major analytical hurdles and have potential applications beyond the KdV setting.

	\item \textit{Robust applicability to physically realistic initial data:}  
	Our analytical framework and numerical algorithm can handle a wide range of physically relevant initial regularities, including singular initial data, nearly white noise random initial conditions, and initial data drawn from Gibbs measures. Numerical experiments convincingly validate the robustness and practical feasibility of our method in these challenging scenarios, successfully capturing subtle dispersive phenomena such as the Talbot-like refocusing of singularities. Thus, our results significantly broaden the applicability of numerical simulations to realistic dispersive PDE models arising in physics and engineering.
\end{itemize}

\

\noindent \textbf{Outline of the article.} The rest of this article is organized as follows. In Section~\ref{sec:main-results}, we introduce the numerical scheme and present the main theoretical results, along with the methodology used to prove them.
In Section~\ref{sec:rescaling}, we rescale the original problem and the numerical scheme into $\lambda$-periodic forms. Additionally, we introduce the methodology for constructing the continuation of numerical solutions, which are initially defined only at discrete time nodes. 
Section~\ref{sec:bourgain} provides the necessary definitions related to Bourgain spaces and presents key preliminary results essential for the analysis.
In Section~\ref{sec:stability}, we establish the uniform boundedness of the numerical method by formulating and analyzing an appropriate fixed point problem within the Bourgain space framework.
Section~\ref{sec:proof_of_main_theorem} focuses on deriving estimates for the consistency error and provides the detailed proof of the main theorem, Theorem~\ref{thm:main}.
Finally, in Section~\ref{sec:numerical_experiments}, we present numerical experiments that validate the theoretical findings and demonstrate the practical performance of the proposed numerical method.

\section{Main results and methodology}\label{sec:main-results}
In this section, we present the main results and methodology of this article, with a focus on the development and analysis of a filtered low-regularity integrator. We state the main theorem, which provides error estimates under suitable choices of the filter parameter, and derive a fully discrete scheme using Fourier spectral methods. Furthermore, we outline the proof of the main results, demonstrating how Bourgain spaces and a  continuous  reformulation are combined to establish the theoretical findings.

\subsection{Numerical scheme and main  results}
We consider the filtered low regularity integrator that was first introduced in \cite{Hofmanova-Schratz-2017}:
\begin{align}\label{eq:filtered_LRI}
	u^{n+1}&=e^{-\tau \partial_x^3}u^n+\frac{1}{2}\int_{\sigma_n}^{\sigma_{n+1}}e^{(t-\sigma_{n+1})\partial_x^3}\Pi_{N}\partial_{x}\left(e^{(\sigma_n-t)\partial_x^3}\Pi_{N}u^n\right)^2\d t,\notag\\
	&=e^{-\tau \partial_x^3}u^n+\frac{1}{6}\Pi_{N}\left(e^{-\tau\partial_x^3}\partial_x^{-1}\Pi_{N}u^n\right)^2-\frac{1}{6}\Pi_{N}e^{-\tau \partial_x^3}\left(\partial_x^{-1}\Pi_{N}u^n\right)^2,
\end{align}
with $u^0=\phi$ and $\sigma_n=n\tau$ for $n=0,1,\cdots,L=T/\tau$ being a partition of the time interval. Here $\Pi_N$ denotes the projection operator for the positive integer $N$, defined by 
\begin{align}
	(\widehat{\Pi_N f})_{k}=\left\{
	\begin{array}{l}
		\hat{f}_{k}\quad \text{if}\;\; |k|\leq N,\\[2mm]
		0\quad \;\, \text{if}\;\; |k|> N,
	\end{array}
	\right.\quad \text{where} \quad \hat{f}_k=\int_0^1 e^{-2\pi ikx}f(x)\d x, \quad \text{for}\quad k\in \mathbb{Z}.
\end{align}
We present the following main theorem of this paper:
\begin{theorem}\label{thm:main}
	For given $T>0$ and initial value $\phi\in H^s(\T)$ with $s\in(-\frac{1}{2},0]$  and $\int_{\T}\phi\,\d x=0$, we denote by $u\in L^{\infty}(0,T;H^{s})$ the unique solution of \eqref{eq:KdV_equation}, and denote by $u^n$, the numerical solution given by \eqref{eq:filtered_LRI}. Then, by choosing $N=O(\tau^{-\frac{1}{2-2s}+})$, there exist constants $\tau_0\in (0,1)$, and $C>0$ such that, for time step size $\tau\in(0,\tau_0)$,
	\begin{align}\label{eq:main_thm}
		\max_{0\leq n\leq T/\tau}\|u(\sigma_n)-u^n\|_{H^{-\frac{1}{2}}}\leq C \tau^{\frac{1+2s}{4-4s}-}.
	\end{align}
\end{theorem}

\begin{remark}
	\upshape 
	Without loss of generality, we can assume that $u(0,x)=\phi(x)$ has zero mean in the spatial variable, i.e., $\int_\T \phi \, \d x = 0$. Otherwise, we can redefine the solution using the transformation
	$$
	\tilde{u}(t,x) := u(t,x - t \, \Pi_0{\phi}) - \Pi_0{\phi},
	$$
	which satisfies \eqref{eq:KdV_equation} with the initial value 
	$$
	\tilde{\phi} = \phi - \Pi_0{\phi},
	$$
	where $\tilde{\phi}$ has zero mean, i.e., $\int_\T \tilde{\phi} \, \d x = 0$.
\end{remark}

By applying Fourier spectral discretization in space, we derive the following fully discrete numerical scheme:
\begin{align}\label{eq:fully_discrete}
	u_N^{n+1} = &\, e^{-\tau \partial_x^3}\Pi_{N_c}u_N^n 
	+\frac{1}{6}\Pi_{N}\left(e^{-\tau\partial_x^3}\partial_x^{-1}\Pi_{N}u_N^n\right)^2
	-\frac{1}{6}\Pi_{N}e^{-\tau \partial_x^3}\left(\partial_x^{-1}\Pi_{N}u_N^n\right)^2,
\end{align}
with initial condition $u_N^0=\Pi_{N_c}\phi$. Here, we set $N_c=C_L N$ with a constant $C_L\geq 1$, which can be suitably chosen to balance the errors arising from the linear and nonlinear discretizations. Note that under the condition $N=O(\tau^{-\frac{1}{2-2s}+})$, we have $N^{-\frac{1}{2}-s}=O(\tau^{\frac{1+2s}{4-4s}-})$. Thus, as a direct consequence of Theorem~\ref{thm:main}, we obtain the following error estimate:

\begin{corollary}\label{rem:fully_discrete}
	Under the conditions of Theorem~\ref{thm:main}, let $u_N^n$ denote the numerical solution of the fully discrete scheme \eqref{eq:fully_discrete} with $N_c=C_L N$, $C_L\geq 1$. Then, under the CFL-type condition $\tau=O(N^{-2+2s-})$, there exist constants $N_0\geq 1$ and $C>0$ such that for all $N\geq N_0$,
	\begin{align}
		\max_{0\leq n\leq T/\tau}\|u(\sigma_n)-u_N^n\|_{H^{-\frac{1}{2}}}\leq C N^{-\frac{1}{2}-s}.
	\end{align}
\end{corollary}

\subsection{Sketch of the proof}
For the readers' convenience, we outline the proof of Theorem~\ref{thm:main} as follows, based on the Bourgain-space estimates at the continuous level, and present the details of the proof in the following sections. 

First, we introduce an appropriate rescaling for both the exact solution $u$ of the original KdV equation \eqref{eq:KdV_equation} and the numerical solution $u^n$ obtained from the filtered low-regularity integrator \eqref{eq:filtered_LRI}:
\begin{align*}
	u_{\lambda}(t,x) = \lambda^{-2}u\left(\frac{t}{\lambda^3}, \frac{x}{\lambda}\right), 
	\quad\text{and}\quad 
	u^n_{\lambda}(x) = \lambda^{-2}u^n\left(\frac{x}{\lambda}\right),
\end{align*}
where $\lambda$ is a constant independent of $N$ and $\tau$. By applying the variation-of-constants formula to the rescaled KdV equation and introducing a suitable temporal localization function $\eta$, we define the nonlinear mapping $\widetilde{\Gamma}$ as follows:
\begin{align}\label{eq:fix_point_2}
	\widetilde{\Gamma}(w)(t) = \eta(t)e^{-t \partial_x^3}u_{\lambda}(0)
	+ \frac{1}{2}\eta(t)\int_{0}^{t} e^{(s - t)\partial_x^3}\eta(s)\partial_x\left(w(s)^2\right)\,\mathrm{d}s,
\end{align}
where $\eta$ is a smooth bump function satisfying $\eta(t)\equiv 1$ for $t\in[0,1]$. It follows from \cite[Proposition 6]{Colliander} that, for sufficiently large $\lambda$, the mapping $\widetilde{\Gamma}$ admits a unique fixed point $u_\lambda^*(t)$ in a suitable Bourgain space, and this fixed point coincides with the rescaled solution:
\begin{align*}
	u_\lambda^*(t)\equiv u_{\lambda}(t),\quad \text{for}\quad t\in[0,1].
\end{align*}

Next, to rigorously analyze the numerical scheme, we construct a continuous-in-time function $\mathcal{U}(t)$ interpolating the discrete numerical solutions $u^n_\lambda$ over the interval $[0,\lambda^3 T]$, i.e.,
\begin{align*}
	\mathcal{U}(t_n) = u^n_{\lambda},\quad \text{for each}\quad n\in\{0,1,\dots,T/\tau\},\quad\text{with}\quad t_n=\lambda^3\sigma_n=n\lambda^3\tau.
\end{align*}
This interpolated function satisfies an integral equation defined in continuous time (see equation~\eqref{eq:continuation_problem} in the next section for the explicit expression). The corresponding nonlinear mapping, incorporating the discretization perturbation, is defined as
\begin{align}\label{eq:fix_point_1}
	\Gamma(w)(t) = \eta(t)e^{-t \partial_x^3}u_{\lambda}^0
	+ \frac{1}{2}\eta(t)\int_{0}^{t} e^{(s - t)\partial_x^3}\eta(s)\partial_x\left(w(s)^2\right)\,\mathrm{d}s
	+ \mathcal{A}(w)(t),
\end{align}
where $\mathcal{A}(w)$ represents the perturbation term resulting from the discretization (see equation~\eqref{eq:def_Gamma} for its explicit form). Clearly, equation \eqref{eq:fix_point_1} can be viewed as a perturbed version of the continuous fixed-point problem \eqref{eq:fix_point_2}. By employing Picard iteration combined with Bourgain-space estimates, we shall prove that, for suitably chosen $\lambda$ and an appropriate relation between the temporal step size $\tau$ and filter constant $N$, the perturbed nonlinear mapping $\Gamma$ admits a unique fixed point $\mathcal{U}^*(t)$ in an appropriate Bourgain space, satisfying
\begin{align*}
	\mathcal{U}^*(t)\equiv \mathcal{U}(t),\quad \text{for}\quad t\in[0,1],
\end{align*}
with Bourgain-space bounds independent of the discretization parameters $\tau$, $N$, and the scaling factor $\lambda$.

By comparing the fixed-point equations \eqref{eq:fix_point_1} and \eqref{eq:fix_point_2}, we transform the estimation of the numerical error $u(\sigma_n)-u^n$ (for $\sigma_n\leq \lambda^{-3}$) into the analysis of the difference between their corresponding fixed points:
\begin{align*}
	\max_{\sigma_n\leq \lambda^{-3}}\|u(\sigma_n)-u^n\|_{H^{-\frac{1}{2}}}
	&\leq \lambda^{\frac{3}{2}}\max_{\lambda^3\sigma_n\leq 1}\|u_{\lambda}(\lambda^3\sigma_n)-u_{\lambda}^n\|_{H^{-\frac{1}{2}}(0,\lambda)}\\
	&\leq \lambda^{\frac{3}{2}}\|u_\lambda^*-\mathcal{U}^*\|_{L^{\infty}(0,1;H^{-\frac{1}{2}}(0,\lambda))}.
\end{align*}
By further applying the embedding theorem for Bourgain spaces (see Lemma~\ref{lem:embedding_Ys}), we can control the $H^{-\frac{1}{2}}$ norm of the difference between $u^{*}_{\lambda}(t)$ and $\mathcal{U}^{*}(t)$ using Bourgain space estimates. In particular, this difference can be bounded by carefully estimating the perturbation term $\mathcal{A}(w)$ appearing in \eqref{eq:fix_point_1}, which captures the error between the two fixed-point formulations in \eqref{eq:fix_point_1} and \eqref{eq:fix_point_2}.

Finally, by iteratively applying the above local-in-time estimates from the initial interval (starting at $n=0$) and sequentially advancing through subsequent intervals, we extend these local error estimates globally in time. Thus, we ultimately establish the rigorous global error bound for all time steps $0\leq n\leq T/\tau$ as stated in Theorem~\ref{thm:main}.

\

\noindent \textbf{Notations.} We denote by $A\lesssim B$ or $B\gtrsim A$ the statement ``$A \leq CB$ for some constant $C > 0$", where $C$ may vary from line to line but it is independent of $\tau$, $N$ and $\lambda$. Similarly, we denote by $A \sim B$ the statement ``$C^{-1}B \leq A \leq CB$ for some constant $C > 0$''. In other words, $A \sim B$ is equivalent to $A \lesssim B \lesssim A$.
The notation $a+$ represents $a + \varepsilon$ for an arbitrarily small $\varepsilon > 0$, and $a-$ represents $a - \varepsilon$ for an arbitrarily small $\varepsilon > 0$. 
We also define the following rescaled quantities:
\begin{align}\label{eq:notation_tau_lambda}
	\tau_\lambda = \lambda^3 \tau, \quad N_\lambda = \lambda^{-1} N, \quad \text{and} \quad t_n = n\tau_\lambda = \lambda^3 \sigma_n,
\end{align}
where $\tau$, $N$, and $\sigma_n$ are the original time step size, degrees of the freedom, and time node, respectively.

\section{Rescaling argument and continuation of the numerical solution}\label{sec:rescaling}

\subsection{The $\lambda$-periodic KdV equation and numerical scheme}

By utilizing the rescaling argument $u_{\lambda}(t,x)=\lambda^{-2}u(\frac{t}{\lambda^3},\frac{x}{\lambda})$ as in \cite{Colliander}, 
we derive that $u_{\lambda}$ is the solution of the following $\lambda$-periodic problem:
\begin{equation}\label{eq:lambda_KdV_equation}
	\left\{\begin{array}{l}
		{\displaystyle \partial_t u_{\lambda}(t,x)+\partial^{3}_{x}u_{\lambda}(t,x)=\frac{1}{2}\partial_x (u_{\lambda}(t,x)^2),\quad x\in[0,\lambda] \;\; \text{and}\;\; t\in(0,\lambda^3 T],}\\[2mm]
		{\displaystyle u_{\lambda}(0)=\phi_{\lambda},}
	\end{array}\right.
\end{equation}
where $\phi_{\lambda}(x)=\lambda^{-2}\phi(\frac{x}{\lambda})$, and $\lambda\geq 1$ is a constant to be determined later, which is independent of $\tau$ and $N$, but may depend on $\|u\|_{L^{\infty}(0,T;H^s)}$.

We define the normalized counting measure $(\d k)_{\lambda}$ on the discrete set $\mathbb{Z}/\lambda:=\{\frac{l}{\lambda} : l\in \mathbb{Z}\}$ as follows:
\begin{align}\label{eq:def_dk_lambda}
	\int a(k)(\d k)_{\lambda}=\frac{1}{\lambda}\sum_{k\in \mathbb{Z}/\lambda}a(k).
\end{align}
The Fourier transform on $[0,\lambda]$ and its inverse are defined by
\begin{align}
	\hat{f}(k)=\int_0^\lambda e^{-2\pi ikx}f(x)\d x,\quad \text{and} \quad f(x)=\int e^{2\pi ikx}\hat{f}(k)(\d k)_{\lambda}.
\end{align}
Additionally, the Fourier transform satisfies the usual properties:
\begin{align}
	\|f\|_{L^2([0,\lambda])}&=\|\hat{f}\|_{L^2((\d k)_{\lambda})}&& (\text{Plancherel's identity})\\
	\int_0^{\lambda}f(x)\overline{g(x)}\d x&=\int\hat{f}(k)\bar{\hat{g}}(k)(\d k)_{\lambda}&& (\text{Parseval's identity})\\
	\widehat{fg}(k)=\hat{f}*_{\lambda}\hat{g}(k)&=\int\hat{f}(k-k_1)\hat{g}(k_1)(\d k_1)_{\lambda}&& (\text{Convolution relation}).
\end{align}
The operator $\partial_x^\alpha$, $\alpha\in \mathbb{N}$, and the Sobolev space $H^s(0,\lambda)$ are defined analogously to the 1-periodic case:
\begin{align}
	\partial_x^\alpha f(x)&=\int e^{2\pi ikx}(2\pi ik)^{\alpha}\hat{f}(k)(\d k)_{\lambda},\quad 
	\|f\|_{H^{s}(0,\lambda)}=\|\hat{f}(k)\langle k\rangle^s\|_{L^2((\d k)_{\lambda})},
\end{align}
where we denote $\langle k\rangle=(1+k^2)^{\frac{1}{2}}$. Thus, we have
\begin{align*}
	\|\phi_{\lambda}\|_{H^s(0,\lambda)}= \lambda^{-\frac12}\Big(\sum_{k\in\mathbb{Z}/\lambda}|\hat{\phi}_{\lambda}(k)|^2\langle k\rangle^{2s}\Big)^{\frac{1}{2}}.
\end{align*}
By direct computation, we find
\begin{align*}
	\hat{\phi}_{\lambda}(k)=\lambda^{-1}\hat{\phi}(\lambda k),
\end{align*}
which implies
\begin{align*}
	\|\phi_{\lambda}\|_{H^s(0,\lambda)}= \lambda^{-\frac12}\Big(\sum_{\xi\in\mathbb{Z}}\big|\lambda^{-1}\hat{\phi}(\xi)\big|^2\Big\langle \frac{\xi}{\lambda}\Big\rangle^{2s}\Big)^{\frac{1}{2}}.
\end{align*}
Noting that $(1+\xi^2)^{\frac{s}{2}} \leq \big(1+\big(\xi/\lambda\big)^2\big)^{\frac{s}{2}} \leq \lambda^{-s}(1+\xi^2)^{\frac{s}{2}}$ for $s\leq 0$ and $\lambda\geq 1$, we conclude that
\begin{align}\label{eq:relation_Hs}
	\lambda^{-\frac{3}{2}}\|\phi\|_{H^s(0,1)}\leq \|\phi_{\lambda}\|_{H^s(0,\lambda)}\leq\lambda^{-\frac{3}{2}-s}\|\phi\|_{H^s(0,1)}.
\end{align}
Thus, for any $\phi\in H^s(0,1)$, we can choose an appropriate $\lambda\geq 1$ to transform the original problem into one with a small initial value. This allows for a closed estimate of the nonlinear terms within a small neighborhood at the initial time.

Using the Fourier transform, the solution operator $e^{-t\partial_{x}^3}$ of the linear homogeneous $\lambda$-periodic problem
\begin{equation}\label{eq:linear_lambda_KdV_equation}
	\left\{
	\begin{array}{l}
		{\displaystyle \partial_t w(t,x)+\partial^{3}_{x}w(t,x)=0, \quad x\in[0,\lambda],}\\[2mm]
		{\displaystyle w(0)=\phi_{\lambda},}
	\end{array}
	\right.
\end{equation}
is expressed as
\begin{align}
	w(t,x)=e^{-t\partial_{x}^3}\phi_{\lambda}=\int e^{2\pi ikx}e^{-(2\pi ik)^3 t}\hat{\phi}_{\lambda}(k)(\d k)_{\lambda}.
\end{align}

For the numerical solution $u^n$ given by \eqref{eq:filtered_LRI}, the corresponding rescaled time discrete scheme is
\begin{align}\label{eq:lambda_filtered_LRI}
	u_{\lambda}^{n+1}=&e^{-\tau_{\lambda} \partial_x^3}u_{\lambda}^n+\frac{1}{2}\int_{t_{n}}^{t_{n+1}}e^{(t-t_{n+1})\partial_x^3}\Pi_{N_{\lambda}}\partial_{x}\left(e^{(t_{n}-t)\partial_x^3}\Pi_{N_{\lambda}}u_{\lambda}^n\right)^2\d t,
\end{align}
with $u^{n}_{\lambda}(x)=\lambda^{-2}u^{n}(\frac{x}{\lambda})$, and
\begin{align}\label{eq:def_Pi_lambda}
	(\widehat{\Pi_{N_{\lambda}} f})_{k}=\left\{
	\begin{array}{l}
		\hat{f}_{k}\quad \text{if}\;\; |k|\leq N_{\lambda},\\[2mm]
		0\quad \;\, \text{if}\;\; |k|> N_{\lambda},
	\end{array}
	\right.\quad\text{for}\quad k\in \mathbb{Z}/\lambda,
\end{align}
where $t_n$, $\tau_{\lambda}$, $N_{\lambda}$ are rescaled quantities of $\sigma_{n}$, $\tau$ and $N$ respectively, as denoted in \eqref{eq:notation_tau_lambda}.

\subsection{Continuation of the numerical solution}

In order to make use of the Bourgain-space estimates at the continuous level, we introduce the time-continuous function $\U(t)$, $t \in [0, \lambda^3 T]$, defined as follows:  
\begin{align}\label{eq:continuation_solution}
	\U(t) &= e^{(t_n-t) \partial_x^3}u_{\lambda}^n + \frac{1}{2} \int_{t_{n}}^{t} e^{(s-t)\partial_x^3} \Pi_{N_{\lambda}} \partial_{x} \left( e^{(t_{n}-s)\partial_x^3} \Pi_{N_{\lambda}} u_{\lambda}^n \right)^2 \, \mathrm{d}s
	\quad\mbox{for\,\, $t \in [t_n, t_{n+1}]$. }
\end{align}  
For the simplicity of notation, we define  
\begin{align}\label{eq:def_F0}
	F^n(t, w) = \frac{1}{2} \int_{t_{n}}^{t} e^{(s-t)\partial_x^3} \Pi_{N_{\lambda}} \partial_{x} \left( e^{(t_{n}-s)\partial_x^3} \Pi_{N_{\lambda}} w \right)^2 \, \mathrm{d}s,
\end{align}  
for each $n \in \{0, 1, \cdots, T/\tau\}$ and $t \in [t_n, t_{n+1}]$. Using this notation, we define the piecewise time-dependent function $\mathcal{F}(t, \U)$ as  
\begin{align}\label{eq:def_F}
	\mathcal{F}(t, \U) = F^n(t, \U(t_n)), \quad \text{for} \quad t \in [t_n, t_{n+1}].
\end{align}  
It follows directly from \eqref{eq:lambda_filtered_LRI} and \eqref{eq:continuation_solution} that  
$\U(t_n) = u^{n}_{\lambda}$,  
and  
\begin{align*}
	\U(t) &= e^{(t_n-t) \partial_x^3} \U(t_n) + \F(t, \U).
\end{align*}  
By substituting $e^{(t_n-s) \partial_x^3} u^n_{\lambda}(t_n) = e^{(t_n-s) \partial_x^3} \U(t_n) = \U(s) - \F(s, \U)$ into the right-hand side of \eqref{eq:continuation_solution}, we obtain  
\begin{align}
	\U(t) &= e^{(t_n-t) \partial_x^3} \U(t_n) + \frac{1}{2} \int_{t_{n}}^{t} e^{(s-t)\partial_x^3} \Pi_{N_{\lambda}} \partial_{x} \left( \Pi_{N_{\lambda}} \U(s) \right)^2 \, \mathrm{d}s \notag \\
	& \quad - \int_{t_{n}}^{t} e^{(s-t)\partial_x^3} \Pi_{N_{\lambda}} \partial_{x} \Big( \Pi_{N_{\lambda}} \F(s, \U) \Pi_{N_{\lambda}} \U(s) \Big) \, \mathrm{d}s \notag \\
	& \quad + \frac{1}{2} \int_{t_{n}}^{t} e^{(s-t)\partial_x^3} \Pi_{N_{\lambda}} \partial_{x} \left( \Pi_{N_{\lambda}} \F(s, \U) \right)^2 \, \mathrm{d}s.
\end{align}  
By iterating this expression, we obtain
\begin{align}\label{eq:continuation_problem}
	\U(t) &= e^{-t \partial_x^3} \U(0) + \frac{1}{2} \int_{0}^{t} e^{(s-t)\partial_x^3} \Pi_{N_{\lambda}} \partial_{x} \left( \Pi_{N_{\lambda}} \U(s) \right)^2 \, \mathrm{d}s \notag \\
	& \quad - \int_{0}^{t} e^{(s-t)\partial_x^3} \Pi_{N_{\lambda}} \partial_{x} \Big( \Pi_{N_{\lambda}} \F(s, \U) \Pi_{N_{\lambda}} \U(s) \Big) \, \mathrm{d}s \notag \\
	& \quad + \frac{1}{2} \int_{0}^{t} e^{(s-t)\partial_x^3} \Pi_{N_{\lambda}} \partial_{x} \left( \Pi_{N_{\lambda}} \F(s, \U) \right)^2 \, \mathrm{d}s\notag\\
	&=e^{-t \partial_x^3} \U(0) + \frac{1}{2} \int_{0}^{t} e^{(s-t)\partial_x^3} \partial_{x} \left(  \U(s) \right)^2 \, \mathrm{d}s+\mathcal{R}(\U)(t),
\end{align}  
where
\begin{align*}
	\mathcal{R}(\U)(t)=&- \int_{0}^{t} e^{(s-t)\partial_x^3} \Pi_{N_{\lambda}} \partial_{x} \Big( \Pi_{N_{\lambda}} \F(s, \U) \Pi_{N_{\lambda}} \U(s) \Big) \, \mathrm{d}s\\
	& + \frac{1}{2} \int_{0}^{t} e^{(s-t)\partial_x^3} \Pi_{N_{\lambda}} \partial_{x} \left( \Pi_{N_{\lambda}} \F(s, \U) \right)^2 \, \mathrm{d}s\\
	&+\frac{1}{2} \int_{0}^{t} e^{(s-t)\partial_x^3}  \partial_{x} \left[\Pi_{N_{\lambda}}\left( \Pi_{N_{\lambda}} \U(s) \right)^2-\left(  \U(s) \right)^2\right] \, \mathrm{d}s. 
\end{align*}
Since the $\lambda$-periodic KdV equation \eqref{eq:lambda_KdV_equation} is equivalent to the integral equation  
\begin{align}\label{eq:mild_lambda_KdV}
	u_{\lambda}(t) = e^{-t \partial_x^3} u_\lambda(0) + \frac{1}{2} \int_{0}^{t} e^{(s-t)\partial_x^3} \partial_{x} \left( u_{\lambda}(s) \right)^2 \, \mathrm{d}s,
\end{align}  
the continuation problem \eqref{eq:continuation_problem} can be viewed as a perturbation of \eqref{eq:mild_lambda_KdV}. By comparing the differences between \eqref{eq:continuation_problem} and \eqref{eq:mild_lambda_KdV}, we can estimate the error $u_{\lambda}(t_n) - u^n_{\lambda}$. Using the rescaling relationship, we have  
\begin{align}\label{eq:relation_scaling}
	\| u(\sigma_n) - u^n \|_{H^{-\frac{1}{2}}(0,1)} \leq \lambda^{\frac{3}{2}} \| u_{\lambda}(t_n) - u^n_{\lambda} \|_{H^{-\frac{1}{2}}(0,\lambda)} = \lambda^{\frac{3}{2}} \| u_{\lambda}(t_n) - \U(t_n) \|_{H^{-\frac{1}{2}}(0,\lambda)}.
\end{align}  
Thus, the proof of Theorem~\ref{thm:main} is reduced to estimating $\| u_{\lambda}(t_n) - \U(t_n) \|_{H^{-\frac{1}{2}}(0,\lambda)}$.

\section{Bourgain spaces and preliminary results}\label{sec:bourgain}
As in the literature \cite{Colliander, Tao}, we define the Bourgain space $X_{s,b}$ for $\lambda$-periodic problems using the following norm:  
\begin{align}
	\|w\|_{X_{s,b}(\mathbb{R} \times ([0,\lambda]))} = 
	\big\|\langle k \rangle^s \langle \sigma - 4\pi^2k^3 \rangle^b \tilde{w}(\sigma,k) \big\|_{L^2(\mathrm{d}\sigma (\mathrm{d}k)_{\lambda})},
\end{align}  
where $\tilde{w}$ is the space-time Fourier transform of $w$, given by  
\begin{align*}
	\tilde{w}(\sigma,k) = \int_{\mathbb{R} \times [0,\lambda]} e^{-2\pi i \sigma t - 2\pi i k x} w(t,x) \, \mathrm{d}t \, \mathrm{d}x.
\end{align*}  
We also introduce a slightly smaller space $Y^s$, compared to $X_{s,\frac{1}{2}}$, as described in \cite{Colliander}. This ensures that any solution in $Y^s$ is also contained in $L^{\infty}_{t}H^{s}_x$:  
\begin{align}
	\|w\|_{Y^s} = \|w\|_{X_{s,\frac{1}{2}}} + \|\langle k \rangle^s \tilde{w}(\sigma,k)\|_{L^2((\mathrm{d}k)_{\lambda}) L^1(\mathrm{d}\sigma)}.
\end{align}  
Additionally, we define the functional space $Z^s$ via the norm  
\begin{align}
	\|w\|_{Z^s} = \|w\|_{X_{s,-\frac{1}{2}}} + \left\|\frac{\langle k \rangle^s \tilde{w}(\sigma,k)}{\langle \sigma - 4\pi^2k^3 \rangle}\right\|_{L^2((\mathrm{d}k)_{\lambda}) L^1(\mathrm{d}\sigma)}.
\end{align}  
Let $\eta \in C^{\infty}_0(\mathbb{R})$ be a smooth bump function such that  
\begin{align}\label{eq:bump_function}
	\text{supp}(\eta) \subset [-2,2], \quad \text{with} \quad \eta \equiv 1 \quad \text{on} \quad [-1,1].
\end{align}

Next, we present some preliminary results (i.e., Lemmas~\ref{lem:free_solution}–\ref{lem:embedding_theorem}), established in \cite{Colliander, Tao}, which will be frequently used throughout this article. These focus on estimates in the Bourgain space for $\lambda$-periodic problems, where the constants are independent of $\lambda$. 

\begin{lemma}[{Free solutions lie in $Y^s$ \cite[Lemma~7.1]{Colliander}}]\label{lem:free_solution}
	Let $\phi_{\lambda}\in H^s([0,\lambda])$ for $s\in\R$, and $\eta$ be the bump function given by \eqref{eq:bump_function}. Then
	\begin{align}\label{eq:free_solution}
		\|\eta(t)e^{-t\partial_x^3}\phi_{\lambda}\|_{Y^s}\leq C\|\phi_{\lambda}\|_{H^s},
	\end{align}
	where $C$ is a positive constant independent of $\lambda$.
\end{lemma}

\begin{lemma}[Embedding result for $Y^s$ \cite{Colliander}]\label{lem:embedding_Ys}
	If $w\in Y^s$ for $s\in \R$, then
	\begin{align}\label{eq:embedding_Ys}
		\|w\|_{L_{t}^{\infty}H_{x}^s}\leq C\|w\|_{Y^s},
	\end{align}
	where $C$ is a positive constant independent of $\lambda$.
\end{lemma}
For the convenience of readers, the proof of this lemma is presented in the supplementary material.

\begin{lemma}[Boundedness of the time localization \cite{Colliander,Tao}]\label{lem:time_localization}
	Let $\eta$ be the bump function given by \eqref{eq:bump_function}. Then
	\begin{align}\label{eq:time_localization}
		\|\eta(t)w\|_{X_{s,b}}\leq C\|w\|_{X_{s,b}}, \quad \|\eta(t)w\|_{Y^s}\leq C\|w\|_{Y^s},\quad \text{and}\quad\|\eta(t)w\|_{Z^s}\leq C\|w\|_{Z^s},
	\end{align}
	where $C$ is a positive constant independent of $\lambda$.
\end{lemma}
For the convenience of readers, the proof of this lemma is presented in the supplementary material.

\begin{lemma}[{Energy estimate \cite[Lemma~7.2]{Colliander}}] \label{lem:energy_estimate}
	Let $\eta$ be the bump function given by \eqref{eq:bump_function}, and $F\in Z^s$ for $s\in\R$. Then
	\begin{align}\label{eq:energy_estimate}
		\left\|\eta(t)\int_0^t e^{(s-t)\partial_x^3}F(s)\d s\right\|_{Y^s}\leq C \|F\|_{Z^s},
	\end{align}
	where $C$ is a positive constant independent of $\lambda$.
\end{lemma}

\begin{lemma}[{Bilinear estimate \cite[Proposition 5]{Colliander}}]\label{lem:bilinear_estimate}
	If $u$ and $v$ are $\lambda$-periodic functions of $x$,
	and have zero $x$-mean for all $t$, then
	\begin{align}
		\|\eta(t)\partial_x (uv)\|_{Z^{-\frac{1}{2}}}\leq C \lambda^{0+}\|u\|_{X_{-\frac{1}{2},\frac{1}{2}}}\|v\|_{X_{-\frac{1}{2},\frac{1}{2}}},
	\end{align}
	where $C$ is a positive constant independent of $\lambda$.
\end{lemma}

\begin{remark}\label{rem:bilinear_estimate}
	It follows from the inequality:
	\begin{align*}
		\langle k_1+k_2\rangle^a\lesssim \langle k_1\rangle^a \langle k_2\rangle^a,
	\end{align*}
	for any $k_1,k_2\in\R$ and $a\geq 0$, that Lemma~\ref{lem:bilinear_estimate} implies, for all $s \geq -\frac{1}{2}$:
	\begin{align}
		\|\eta(t)\partial_x (uv)\|_{Z^{s}}\leq C\lambda^{0+}\|u\|_{X_{s,\frac{1}{2}}}\|v\|_{X_{s,\frac{1}{2}}},
	\end{align}
	where $C$ is a positive constant independent of $\lambda$.
\end{remark}

\begin{lemma}[{Embedding theorem for $X_{0,\frac{1}{3}}$ \cite[Lemma~7.3]{Colliander}}]\label{lem:embedding_theorem}
	Let $w\in X_{0,\frac{1}{3}}$ be a $\lambda$-periodic function and let $\eta$ be the bump function given by \eqref{eq:bump_function}. Then
	\begin{align}\label{eq:embedding_theorem}
		\|\eta(t)w\|_{L^4_t L^4_x}\leq C\|w\|_{X_{0,\frac{1}{3}}},
	\end{align}
	where $C$ is a positive constant independent of $\lambda$.
\end{lemma}

\begin{corollary}\label{corollary:Z0}
	If $w\in Z^0$ is a $\lambda$-periodic function, then
	\begin{align}
		\|w\|_{Z^0}\leq C\|w\|_{L^{\frac{4}{3}}_t L^{\frac{4}{3}}_x},
	\end{align}
	where $C$ is a positive constant independent of $\lambda$.
\end{corollary}
For the convenience of readers, the proof of this corollary is presented in the supplementary material.

\begin{lemma}[Bernstein's type inequality]\label{lem:Bernstein}
	Let $ \Pi_{N_{\lambda}} $ be the projection operator defined by \eqref{eq:def_Pi_lambda}, and let $ \Pi_{>N_{\lambda}} = I - \Pi_{N_{\lambda}} $. Then, the following inequalities hold:
	\begin{equation}\label{eq:Bernstein_1}
		\begin{array}{l}
			\|\Pi_{N_{\lambda}} u\|_{H^{\gamma}} \leq C N_{\lambda}^{\gamma-s} \|u\|_{H^s}, \quad \|\Pi_{N_{\lambda}} u\|_{X_{\gamma,b}} \leq C N_{\lambda}^{\gamma-s} \|u\|_{X_{s,b}}, \\[2mm]
			\|\Pi_{N_{\lambda}} u\|_{Y^{\gamma}} \leq C N_{\lambda}^{\gamma-s} \|u\|_{Y^s}, \quad \|\Pi_{N_{\lambda}} u\|_{Z^{\gamma}} \leq C N_{\lambda}^{\gamma-s} \|u\|_{Z^s},
		\end{array}
	\end{equation}
	for all $ s \leq \gamma $, $ b \in \mathbb{R} $, and
	\begin{align}\label{eq:Bernstein_3}
		\begin{array}{l}
			\|\Pi_{>N_{\lambda}} u\|_{H^{\gamma}} \leq C N_{\lambda}^{\gamma-s} \|u\|_{H^s}, \quad \|\Pi_{>N_{\lambda}} u\|_{X_{\gamma,b}} \leq C N_{\lambda}^{\gamma-s} \|u\|_{X_{s,b}}, \\[2mm]
			\|\Pi_{>N_{\lambda}} u\|_{Y^{\gamma}} \leq C N_{\lambda}^{\gamma-s} \|u\|_{Y^s}, \quad \|\Pi_{>N_{\lambda}} u\|_{Z^{\gamma}} \leq C N_{\lambda}^{\gamma-s} \|u\|_{Z^s},
		\end{array}
	\end{align}
	for all $ \gamma \leq s $, $ b \in \mathbb{R} $, and
	\begin{align}\label{eq:Bernstein_2}
		\|\Pi_{N_{\lambda}} u\|_{L^{4}} \leq C N_{\lambda}^{\frac{1}{4}} \|u\|_{L^2},
	\end{align}
	with a positive constant $ C $ independent of $ \lambda $.
\end{lemma}
For the convenience of readers, the proof of this lemma is presented in the supplementary material.

\section{Uniform boundedness of the numerical solution}\label{sec:stability}

Recalling the continuation problem for the numerical solution in \eqref{eq:continuation_problem}, we note that the Bourgain norm--a space-time norm combining spatial and temporal regularity--requires localizing the right-hand side of \eqref{eq:continuation_problem} in time. This localization is important for estimating both $\U$ and the numerical error $u_{\lambda}-\U$ within the Bourgain framework. For this purpose, we introduce the following nonlinear operator:

\begin{align}\label{eq:def_Gamma}
	\Gamma(w)(t)=\eta(t)e^{-t \partial_x^3}\phi_{\lambda}+\frac{1}{2}\eta(t)\int_{0}^{t}e^{(s-t)\partial_x^3}\eta(s)\partial_{x}\left(w(s)\right)^2\d s+\mathcal{A}(w)(t),
\end{align}
where we denote
\begin{align*}
	\mathcal{A}(w)(t)&=\mathcal{A}_1(w)(t)+\mathcal{A}_2(w)(t)+\mathcal{A}_3(w)(t),\\
	\mathcal{A}_1(w)(t)&:=-\eta(t)\int_{0}^{t}e^{(s-t)\partial_x^3}\eta(s)\Pi_{N_{\lambda}}\partial_{x}\Big(\Pi_{N_{\lambda}}\F(s,w)\Pi_{N_{\lambda}}w(s)\Big)\d s,\\
	\mathcal{A}_2(w)(t)&:=\frac{1}{2}\eta(t)\int_{0}^{t}e^{(s-t)\partial_x^3}\eta(s)\Pi_{N_{\lambda}}\partial_{x}\left(\Pi_{N_{\lambda}}\F(s,w)\right)^2\d s,\\
	\mathcal{A}_3(w)(t)&:=\frac{1}{2}\eta(t)\int_{0}^{t}e^{(s-t)\partial_x^3}\eta(s)\partial_{x}\left[\Pi_{N_{\lambda}}\left(\Pi_{N_{\lambda}}w(s)\right)^2-\left(w(s)\right)^2\right]\d s.
\end{align*}

Since $\eta(t) \equiv 1$ for $t \in [-1,1]$ (see \eqref{eq:bump_function}), comparing the continuation problem \eqref{eq:continuation_problem} with \eqref{eq:def_Gamma} yields the following key relationship:


\begin{remark}\label{lem:relation_Gamma}
	\upshape
	For fixed $N_{\lambda}$ and $\tau_{\lambda}$, suppose $\U^* \in Y^s$ for $s \in \left( -\frac{1}{2}, 0 \right]$ is a fixed point of $\Gamma$ (i.e., $\Gamma(\U^*) = \U^*$). 
	Then, $\U^*(t) \equiv \U(t)$ for $t \in [0, 1]$, where $\U$ is defined in \eqref{eq:continuation_solution}.
\end{remark}

We will demonstrate the existence of a fixed point for $\Gamma$ within the Bourgain space $Y^s$. To this end we construct the iterative sequence $\U^\ell$ in the following way:
\begin{align}\label{eq:iterative_sequence}
	\U^0=0, \quad \text{and}\quad \U^{\ell+1}=\Gamma(\U^\ell),\quad \text{for}\quad \ell\geq 0.
\end{align}
Using the Bourgain estimate, we establish the following boundedness results for the iterative sequence.
%

\begin{lemma}\label{lem:boundedness_um}
	Let $\Gamma$ be the nonlinear operator defined in \eqref{eq:def_Gamma}, with $\phi_{\lambda} \in H^{s}(0, \lambda)$, and $N_{\lambda}, \tau_{\lambda} > 0$ being the rescaled quantities of $\phi \in H^s(0, 1)$ and $N, \tau > 0$ defined in \eqref{eq:lambda_KdV_equation} and \eqref{eq:lambda_filtered_LRI}. Then, for fixed $N$ and $\tau$ satisfying $N \lesssim \tau^{-\frac{1}{2 - 2s}}$ for $s \in \left(-\frac{1}{2}, 0\right]$, there exists a positive constant $\lambda_0$ independent of $N$ and $\tau$, but possibly dependent on $\|\phi\|_{H^s}$, such that for any $\lambda \geq \lambda_0$, the following bound holds:
	\begin{align}\label{eq:boundedness_um}
		\|\U^\ell\|_{Y^s} \leq C^{\#} \|\phi_{\lambda}\|_{H^s},
	\end{align}
	for all $\ell \geq 0$, where $C^{\#}$ is a positive constant independent of $\tau$, $N$, and $\lambda$.
\end{lemma}

\begin{proof}
For better readability, we provide the detailed proof of this lemma to Appendix~\ref{Appendix:B} in the supplementary material.
\end{proof}

We are now ready to demonstrate  the existence of the fixed point of $\Gamma$ in the Bourgain space $Y^s$:
\begin{theorem}\label{thm:existence}
	Under the conditions of Lemma~\ref{lem:boundedness_um}, there exists a positive constant $\lambda_0$ independent of $N$ and $\tau$ but possibly dependent on $\|\phi\|_{H^s}$ such that for any $\lambda\geq \lambda_0$ the equation $\Gamma(\U^*)=\U^*$ admits a solution $\U^*\in Y^s$ satisfying 
	\begin{align}\label{eq:boundedness_u*}
		\|\U^*\|_{Y^s}\leq C^{\#}\|\phi_{\lambda}\|_{H^s},
	\end{align}
	where $C^{\#}$ is a positive constant independent of $\tau$, $N$ and $\lambda$.
\end{theorem}

\begin{proof}
	Let $\{\U^\ell\}_{\ell \geq 0}$ be the iterative sequence given by \eqref{eq:iterative_sequence}. 
	Then, for $\ell\geq 2$, we have
	\begin{align}
		\U^{\ell+1}-\U^{\ell}=&\frac{1}{2}\eta(t)\int_{0}^{t}e^{(s-t)\partial_x^3}\eta(s)\Pi_{N_{\lambda}}\partial_{x}\Big(\Pi_{N_{\lambda}}\left(\U^\ell-\U^{\ell-1}\right)(s)\cdot\Pi_{N_{\lambda}}\left(\U^\ell+\U^{\ell-1}\right)(s)\Big)\d s\notag\\
		&+\mathcal{A}_1(\U^\ell)-\mathcal{A}_1(\U^{\ell-1})+\mathcal{A}_2(\U^\ell)-\mathcal{A}_2(\U^{\ell-1}),
	\end{align} 
	where we have used the equivalent reformulation of \eqref{eq:def_Gamma}, namely \eqref{eq:def_Gamma_equivalent}, as introduced in the supplementary material.
	
	From Lemma~\ref{lem:boundedness_um}, we know that the sequence is uniformly bounded in the $Y^s$ space, i.e., $\|\U^\ell\|_{Y^s}\leq C^{\#}\|\phi_{\lambda}\|_{H^s}$ for all $\ell\geq 0$.
	Combining the boundedness with Lemma~\ref{lem:energy_estimate},  Lemma~\ref{lem:bilinear_estimate}, and Remark~\ref{rem:bilinear_estimate}, we deduce that
	\begin{align}
		&\left\|\frac{1}{2}\eta(t)\int_{0}^{t}e^{(s-t)\partial_x^3}\eta(s)\Pi_{N_{\lambda}}\partial_{x}\Big(\Pi_{N_{\lambda}}\left(\U^\ell-\U^{\ell-1}\right)(s)\cdot\Pi_{N_{\lambda}}\left(\U^\ell+\U^{\ell-1}\right)(s)\Big)\d s\right\|_{Y^s}\notag\\
		&\quad \leq C\lambda^{0+}\left\|\U^\ell-\U^{\ell-1}\right\|_{Y^s}\left(\left\|\U^\ell\right\|_{Y^s}+\left\|\U^{\ell-1}\right\|_{Y^s}\right)\notag\\
		&\quad\leq 2C\cdot C^{\#}\lambda^{-\frac{3}{2}-s+}\|\phi\|_{H^s}\left\|\U^\ell-\U^{\ell-1}\right\|_{Y^s}.
	\end{align}
	Since $\mathcal{A}_1(w)$ and $\mathcal{A}_2(w)$ are multi-linear functions of $w$, we can estimate $\|\mathcal{A}_1(\U^\ell)-\mathcal{A}_1(\U^{\ell-1})\|_{Y^s}$ and $\|\mathcal{A}_2(\U^\ell)-\mathcal{A}_2(\U^{\ell-1})\|_{Y^s}$ in a same manner as in the proof of Lemma~\ref{lem:boundedness_um}:
	\begin{align}
		&\|\mathcal{A}_1(\U^\ell)-\mathcal{A}_1(\U^{\ell-1})\|_{Y^s}\notag\\
		&\leq C\tau_{\lambda}N_{\lambda}^{2-2s}(C^{\#})^2\|\phi_{\lambda}\|^2_{H^s}\left\|\U^\ell-\U^{\ell-1}\right\|_{Y^s}\notag\\
		&\quad+C\tau_{\lambda}^\frac{4}{3} N_{\lambda}^{\frac{5}{2}-2s}C^{\#}\|\phi_{\lambda}\|^2_{H^s}\Big(\|\U^{\ell-1}-\U^{\ell-2}\|_{Y^s}+\left\|\U^{\ell}-\U^{\ell-1}\right\|_{Y^s}\Big)\notag\\
		&\quad+C\tau_{\lambda}^\frac{5}{3} N_{\lambda}^{\frac{5}{2}-2s}\|\phi_{\lambda}\|_{H^s}^2\left\|\U^{\ell}-\U^{\ell-1}\right\|_{Y^s}\notag\\
		&\quad+C\sum_{j=2}^{7}\tau_{\lambda}^{j}N_{\lambda}^{[1+\frac{3j}{2}-(j+1)s]}(C^{\#})^{j+1}\|\phi_{\lambda}\|_{H^s}^{j+1}\Big(\|\U^{\ell-1}-\U^{\ell-2}\|_{Y^s}+\left\|\U^{\ell}-\U^{\ell-1}\right\|_{Y^s}\Big)\notag\\
		&=:\tilde{\mathcal{C}}\left(s,\tau_{\lambda},N_{\lambda},\left\|\U^{\ell}\right\|_{Y^s},\|\U^{\ell-1}\|_{Y^s}, \|\U^{\ell-2}\|_{Y^s}\right),
	\end{align}
	and
	\begin{align*}
		&\|\mathcal{A}_2(\U^\ell)-\mathcal{A}_2(\U^{\ell-1})\|_{Y^s}\\
		&\quad\leq C\tau_{\lambda}N_{\lambda}^{\frac{7}{4}-s}C^{\#}\left\|\phi_{\lambda}\right\|_{H^s}\tilde{\mathcal{C}}\left(s,\tau_{\lambda},N_{\lambda},\left\|\U^{\ell}\right\|_{Y^s},\|\U^{\ell-1}\|_{Y^s}, \|\U^{\ell-2}\|_{Y^s}\right).
	\end{align*}
	Under condition $N\lesssim \tau^{-\frac{1}{2-2s}}$, it  follows from the same discussions as in \eqref{eq:bound_constant_1}--\eqref{eq:bound_constant_3} that there exists $\lambda_0>1$ independent of $\tau$ and $N$, such that for $\lambda>\lambda_0$, the following inequality holds
	\begin{align*}
		\|\U^{\ell+1}-\U^{\ell}\|_{Y^s}\leq \frac{1}{4}\|\U^{\ell}-\U^{\ell-1}\|_{Y^s}+\frac{1}{4}\|\U^{\ell-1}-\U^{\ell-2}\|_{Y^s}.
	\end{align*}
	Therefore, the sequence $\{\U^\ell\}_{\ell \geq 0}$ is a Cauchy sequence in $Y^s$, and it converges to a fixed point $\U^*$ of the operator $\Gamma$.
\end{proof}

\section{Error estimate of the filtered low-regularity integrator \eqref{eq:filtered_LRI}}\label{sec:proof_of_main_theorem}
From \cite[Proposition~6]{Colliander}, we know that there exists a constant $\lambda_0>1$ such that when $\lambda\geq \lambda_0$, the  equation 
\begin{align}\label{eq:local_well_posedness}
	u_{\lambda}^*(t)=&\eta(t)e^{-t \partial_x^3}\phi_{\lambda}+\frac{1}{2}\eta(t)\int_{0}^{t}e^{(s-t)\partial_x^3}\eta(s)\partial_{x}\left(u_{\lambda}^{*}(s)\right)^2\d s,
\end{align}
has a unique solution in the $Y^s$ space and satisfies 
$$\|u_{\lambda}^*\|_{Y^s}\leq C^{\#}\|\phi_{\lambda}\|_{H^s}.$$
Since the bump function $\eta\equiv 1$ for $t \in [0, 1]$, we obtain the relation
\begin{align}\label{eq:relation_u*_u}
	u_{\lambda}^*(t)=u_{\lambda}(t),\quad \text{for}\quad t \in [0, 1].
\end{align}
Therefore, by combining Remark~\ref{lem:relation_Gamma} with Lemma~\ref{lem:embedding_Ys}, we have
\begin{align}\label{eq:u-u*}
	\|u_{\lambda}-\U\|_{L^{\infty}(0,1; H^{-\frac{1}{2}})}\lesssim \|u_{\lambda}^{*}-\U^*\|_{Y^{-\frac{1}{2}}},
\end{align}
where $\U^*$ is the solution to
\begin{align}\label{eq:Gamma_U*}
	\Gamma(\U^*)=\U^*,
\end{align}
constructed in the Theorem~\ref{thm:existence}.
Thus, we can transform the error estimate between $u(t_n)$ and $u^{n}$ into an estimate of the Bourgain norm of the error
\begin{align}\label{eq:def_error}
	e(t,x):=u_{\lambda}^{*}(t,x)-\U^*(t,x).
\end{align}

Subtracting \eqref{eq:Gamma_U*} from \eqref{eq:local_well_posedness}, we derive the error equation
\begin{align}\label{eq:error_function}
	u_{\lambda}^*(t)-\U^*(t)=&\eta(t)e^{-t \partial_x^3}\big(u_{\lambda}^*(0)-\U^*(0)\big)\notag\\
	&+\frac{1}{2}\eta(t)\int_{0}^{t}e^{(s-t)\partial_x^3}\eta(s)\partial_{x}\left[\left(u_{\lambda}^{*}(s)\right)^2-\left(\U^{*}(s)\right)^2\right]\d s-\mathcal{A}(\U^*)(t),
\end{align}
where $\mathcal{A}(\U^*)(t)$ is defined in \eqref{eq:def_Gamma}.

Taking the $Y^{-\frac{1}{2}}$ norm of the both sides of \eqref{eq:error_function}, we deduce that
\begin{align}\label{eq:Y_norm_inequality}\notag
	\|e\|_{Y^{-\frac{1}{2}}}&\leq \left\|\frac{1}{2}\eta(t)\int_{0}^{t}e^{(s-t)\partial_x^3}\eta(s)\partial_{x}\Big(e(s)\cdot\left(u_{\lambda}^{*}+\U^*\right)(s)\Big)\d s\right\|_{Y^{-\frac{1}{2}}}\\
	&\quad + \left\|\eta(t)e^{-t \partial_x^3}e(0)\right\|_{Y^{-\frac{1}{2}}}
	+\|\mathcal{A}(\U^*)\|_{Y^{-\frac{1}{2}}}.
\end{align}

From Lemmas~\ref{lem:free_solution}, \ref{lem:energy_estimate} and \ref{lem:bilinear_estimate}, we know that
\begin{align}\label{eq:e_0_Y12}
	\left\|\eta(t)e^{-t \partial_x^3}e(0)\right\|_{Y^{-\frac{1}{2}}}\lesssim \|e(0)\|_{H^{-\frac{1}{2}}},
\end{align} 
and
\begin{align}\label{eq:e_times_u_Y12}
	&\left\|\frac{1}{2}\eta(t)\int_{0}^{t}e^{(s-t)\partial_x^3}\eta(s)\partial_{x}\Big(e(s)\cdot\left(u_{\lambda}^{*}+\U^*\right)(s)\Big)\d s\right\|_{Y^{-\frac{1}{2}}}\notag\\
	&\quad\lesssim \left\|\eta(t)\partial_{x}\Big(e(t)\cdot\left(u_{\lambda}^{*}+\U^*\right)(t)\Big)\right\|_{Z^{-\frac{1}{2}}} \lesssim \lambda^{0+}\left(\left\|u^{*}_{\lambda}\right\|_{Y^{-\frac{1}{2}}}+\left\|\U^{*}\right\|_{Y^{-\frac{1}{2}}}\right)\cdot\|e\|_{Y^{-\frac{1}{2}}}.
\end{align}
We shall estimate the remaining term on the right-hand side of \eqref{eq:Y_norm_inequality}, i.e. $\|\mathcal{A}(\U^*)\|_{Y^{-\frac{1}{2}}}$.

\subsection{Consistency error}

The remaining term in \eqref{eq:Y_norm_inequality}, namely $\|\mathcal{A}(\U^*)\|_{Y^{-\frac{1}{2}}}$, can be viewed as the consistency errors of the numerical scheme \eqref{eq:lambda_filtered_LRI} when solving the rescaled KdV equation after applying the rescaling transformation. Recalling the definition of $\mathcal{A}(\U^*)$ in \eqref{eq:def_Gamma}, we note that $\mathcal{A}(\U^*)$ consists of three components: $\mathcal{A}_1(\U^*)$, $\mathcal{A}_2(\U^*)$, and $\mathcal{A}_3(\U^*)$. To estimate these three consistency error terms, we provide the following two lemmas.

\begin{lemma}\label{lem:estimate_A1A2}
	Under the assumptions of Theorem~\ref{thm:existence}, the following estimate holds:
	\begin{align}\label{eq:estimate_A1A2}
		\|\mathcal{A}_1(\U^*)\|_{Y^{-\frac{1}{2}}} + \|\mathcal{A}_2(\U^*)\|_{Y^{-\frac{1}{2}}} \leq C \lambda^{-\frac{3}{2}} \tau N^{3\left( \frac{1}{2} - s \right)},
	\end{align}
	where $C$ is a positive constant independent of $N$, $\tau$, and $\lambda$, but possibly dependent on $\|\phi\|_{H^s}$.
\end{lemma} 
\begin{proof}
	Recalling the definitions of $\mathcal{A}1$ and $\mathcal{A}2$ in \eqref{eq:def_Gamma}, and applying Lemma~\ref{lem:energy_estimate}, Lemma~\ref{lem:embedding_theorem}, and Corollary~\ref{corollary:Z0}, we derive the following estimates:
	\begin{align}
		\|\mathcal{A}_1(\U^*)\|_{Y^{-\frac{1}{2}}}&\lesssim \left\|\eta(t)\Pi_{N_{\lambda}}\partial_{x}\Big(\Pi_{N_{\lambda}}\F(t,\U^*)\Pi_{N_{\lambda}}\U^*(t)\Big)\right\|_{Z^{-\frac{1}{2}}}\notag\\
		&\lesssim N_{\lambda}^{\frac{1}{2}}\left\|\eta(t)\Big(\Pi_{N_{\lambda}}\F(t,\U^*)\Pi_{N_{\lambda}}\U^*(t)\Big)\right\|_{Z^0}\notag\\
		&\lesssim N_{\lambda}^{\frac{1}{2}}\left\|\eta(t)\Big(\Pi_{N_{\lambda}}\F(t,\U^*)\Pi_{N_{\lambda}}\U^*(t)\Big)\right\|_{L^{\frac{4}{3}}_tL^{\frac{4}{3}}_x}\notag\\
		&\lesssim N_{\lambda}^{\frac{1}{2}}\left\|\eta(t)\Pi_{N_{\lambda}}\F(t,\U^*)\right\|_{L^2_t L^2_x}\|\Pi_{N_{\lambda}}\U^*(t)\|_{L^4_t L^4_x},
	\end{align}
	and
	\begin{align}
		\|\mathcal{A}_2(\U^*)\|_{Y^{-\frac{1}{2}}}&\lesssim N_{\lambda}^{\frac{1}{2}}\left\|\eta(t)\Pi_{N_{\lambda}}\F(t,\U^*)\right\|_{L^2_t L^2_x}\|\eta(t)\Pi_{N_{\lambda}}\F(t,\U^*)\|_{L^4_t L^4_x}.
	\end{align}
	Next, we estimate $\|\eta(t)\Pi_{N_{\lambda}}\F(t,\U^*)\|_{L^2_t L^2_x}$, $\|\eta(t)\Pi_{N_{\lambda}}\F(t,\U^*)\|_{L^4_t L^4_x}$, and $\|\Pi_{N_{\lambda}} \U^*(t)\|_{L^4_t L^4_x}$ in the same way as in the proof of Lemma~\ref{lem:boundedness_um}. This leads to 
	\begin{align}\label{eq:A_1_Y12}
		\|\mathcal{A}_1(\U^*)\|_{Y^{-\frac{1}{2}}}
		&\lesssim \tau_{\lambda}N_{\lambda}^{\frac{3}{2}-3s}\left\|\U^{*}\right\|^3_{Y^s}+\tau_{\lambda}^\frac{4}{3} N_{\lambda}^{2-3s}\|\phi_{\lambda}\|_{H^s}\|\U^{*}\|^2_{Y^s}\notag\\
		&\quad+\tau_{\lambda}^\frac{5}{3} N_{\lambda}^{2-3s}\|\phi_{\lambda}\|_{H^s}^2\left\|\U^{*}\right\|_{Y^s}+\sum_{j=2}^{7}\tau_{\lambda}^{j}N_{\lambda}^{[\frac{1}{2}+\frac{3j}{2}-(j+2)s]}\|\U^{*}\|^{j+2}_{Y^s}\notag\\
		&=:\mathcal{C}^*\left(s,\tau_{\lambda},N_{\lambda},\left\|\U^{*}\right\|_{Y^s}\right).\\
		\label{eq:A_2_Y12}
		\|\mathcal{A}_2(\U^*)\|_{Y^{-\frac{1}{2}}}&\lesssim\tau_{\lambda}N_{\lambda}^{\frac{7}{4}-s}\left\|\U^{*}\right\|_{Y^s}\cdot\mathcal{C}^*\left(s,\tau_{\lambda},N_{\lambda},\left\|\U^{*}\right\|_{Y^s}\right).
	\end{align}
	Substituting $\|\U^{*}\|_{Y^s}\leq C^{\#}\|\phi_{\lambda}\|_{H^s}$, $\|\phi_{\lambda}\|_{H^s}\leq\lambda^{-\frac{3}{2}-s}\|\phi\|_{H^s}$, $\tau_{\lambda}=\lambda^3\tau$ and $N_{\lambda}=\lambda^{-1}N$, we get the following estimates:
	\begin{align}\label{eq:estimate_constant}
		\tau_{\lambda}N_{\lambda}^{\frac{7}{4}-s}\left\|\U^{*}\right\|_{Y^s}&\leq \left[C^{\#}\lambda^{-\frac{1}{4}}\tau N^{\frac{7}{4}-s}\right]\cdot\|\phi\|_{H^s}\lesssim 1,
	\end{align}
	and
	\begin{align}\label{eq:estimate_C*}
		&\mathcal{C}^*\left(s,\tau_{\lambda},N_{\lambda},\left\|\U^{*}\right\|_{Y^s}\right)\notag\\		
		&\leq\lambda^{-\frac{3}{2}}\tau N^{3(\frac{1}{2}-s)}\cdot\Big((C^{\#})^3\lambda^{-\frac{3}{2}}\|\phi\|_{H^s}^3+(C^{\#})^2\lambda^{-1}\tau^{\frac{1}{3}}N^{\frac{1}{2}}\|\phi\|_{H^s}^3\notag\\
		&\qquad+C^{\#}\tau^{\frac{2}{3}}N^{\frac{1}{2}}\|\phi\|_{H^s}^3+\sum_{j=2}^{7}(C^{\#})^{j+2}\lambda^{-2}\tau^{j-1}N^{[\frac{1}{2}+\frac{3(j-1)}{2}-(j-1)s]}\|\phi\|_{H^s}^{j+2}\Big)\notag\\
		&\lesssim \lambda^{-\frac{3}{2}}\tau N^{3(\frac{1}{2}-s)},
	\end{align}
	which holds for $\lambda\geq 1$ and $N\lesssim\tau^{-\frac{1}{2-2s}}$ with $s\in(-\frac{1}{2},0]$.
	In above, we have used the inequality
	\begin{align*}
		j-1\geq \frac{1}{2-2s}\cdot\left[\frac{1}{2}+\frac{3(j-1)}{2}-(j-1)s\right],
	\end{align*}
	which  is equivalent to $(\frac{1}{2}-s)(j-1)\geq \frac{1}{2}$ that holds  for $j\geq 2$ and $s\leq 0$. Combining \eqref{eq:A_1_Y12}--\eqref{eq:estimate_C*}, we complete the proof. 
\end{proof}

\begin{lemma}\label{lem:estimate_R1}
	Let $u^{*}_{\lambda}$ be the unique solution of \eqref{eq:local_well_posedness} in $Y^s$ for $s\in(-\frac{1}{2},0]$, and let $\U^*$ be the fixed point of the operator $\Gamma$ under the same assumptions as in Theorem~\ref{thm:existence}. Then, the following result holds:
	\begin{align}\label{eq:estimate_R1}
		\|\mathcal{A}_3(\U^*)\|_{Y^{-\frac{1}{2}}}\leq  C\left[\lambda^{-\frac{5}{2}-s+}N^{-\frac{1}{2}-s}+\lambda^{0+}\left(\left\|u^{*}_{\lambda}\right\|_{Y^{-\frac{1}{2}}}+\left\|\U^{*}\right\|_{Y^{-\frac{1}{2}}}\right)\cdot\|e\|_{Y^{-\frac{1}{2}}}\right],
	\end{align}
	where $C$ is a positive constant independent of $N$, $\tau$ and $\lambda$, but possibly dependent on $\|\phi\|_{H^s}$.
\end{lemma}

\begin{proof}
	We decompose $\mathcal{A}_3(\U^*)$ as follows
	\begin{align}\label{eq:A_3}
		\mathcal{A}_3(\U^*)=&\frac{1}{2}\eta(t)\int_{0}^{t}e^{(s-t)\partial_x^3}\eta(s)\partial_{x}\left[\Pi_{N_{\lambda}}\left(\Pi_{N_{\lambda}}\U^*(s)\right)^2-\Pi_{N_{\lambda}}\left(\Pi_{N_{\lambda}}u^*_{\lambda}(s)\right)^2\right]\d s\notag\\
		&+\frac{1}{2}\eta(t)\int_{0}^{t}e^{(s-t)\partial_x^3}\eta(s)\partial_{x}\left[\Pi_{N_{\lambda}}\left(\Pi_{N_{\lambda}}u^*_{\lambda}(s)\right)^2-\left(u^*_{\lambda}(s)\right)^2\right]\d s\notag\\
		&+\frac{1}{2}\eta(t)\int_{0}^{t}e^{(s-t)\partial_x^3}\eta(s)\partial_{x}\left[\left(u^*_{\lambda}(s)\right)^2-\left(\U^*(s)\right)^2\right]\d s\notag\\
		=&:\mathcal{B}_1(\U^*,u^*_{\lambda})+\mathcal{B}_2(u^*_{\lambda})+\mathcal{B}_3(\U^*,u^*_{\lambda}).
	\end{align}
	Using the same arguments as in \eqref{eq:e_times_u_Y12}, and noting the boundedness of the projection operator, we obtain:
	\begin{align}\label{eq:estimate_B1B3}
		\left\|\mathcal{B}_1(\U^*,u^*_{\lambda})\right\|_{Y^{-\frac{1}{2}}}+\left\|\mathcal{B}_3(\U^*,u^*_{\lambda})\right\|_{Y^{-\frac{1}{2}}}\lesssim \lambda^{0+}\left(\left\|u^{*}_{\lambda}\right\|_{Y^{-\frac{1}{2}}}+\left\|\U^{*}\right\|_{Y^{-\frac{1}{2}}}\right)\cdot\|e\|_{Y^{-\frac{1}{2}}}.
	\end{align}
	Next, we further split $\mathcal{B}_2(u^*_{\lambda})$ as follows
	\begin{align}\label{eq:R_1}
		\mathcal{B}_2(u^*_{\lambda})=&-\frac{1}{2}\eta(t)\int_{0}^{t}e^{(s-t)\partial_x^3}\eta(s)\Pi_{>N_{\lambda}}\partial_{x}\left(u_{\lambda}^{*}(s)\right)^2\d s\notag\\
		&-\frac{1}{2}\eta(t)\int_{0}^{t}e^{(s-t)\partial_x^3}\eta(s)\Pi_{N_{\lambda}}\partial_{x}\Big[\Pi_{>N_{\lambda}}u^*_{\lambda}(s)\cdot\left(u^*_{\lambda}(s)+\Pi_{N_{\lambda}}u^*_{\lambda}(s)\right)\Big]\d s \notag\\
		=&:\mathcal{B}_{21}(u^*_{\lambda})+\mathcal{B}_{22}(u^*_{\lambda}),
	\end{align}
	where we denote $\Pi_{>N_{\lambda}}=I-\Pi_{N_{\lambda}}$. By using Lemmas~\ref{lem:energy_estimate}, \ref{lem:bilinear_estimate}, \ref{lem:Bernstein} and Remark~\ref{rem:bilinear_estimate}, we get
	\begin{align}\label{eq:R_11}
		\left\|\mathcal{B}_{21}(u^*_{\lambda})\right\|_{Y^{-\frac{1}{2}}}&\lesssim 	\left\|\eta(t)\Pi_{>N_{\lambda}}\partial_{x}\left(u_{\lambda}^{*}(t)\right)^2\right\|_{Z^{-\frac{1}{2}}}\notag\\
		&\lesssim N_{\lambda}^{-\frac{1}{2}-s}\left\|\eta(t)\partial_{x}\left(u_{\lambda}^{*}(t)\right)^2\right\|_{Z^{s}}\lesssim \lambda^{0+}N_{\lambda}^{-\frac{1}{2}-s}\|{u^*_{\lambda}}\|_{Y^s}^2,
	\end{align}
	and
	\begin{align}\label{eq:R_12}
		\left\|\mathcal{B}_{22}(u^*_{\lambda})\right\|_{Y^{-\frac{1}{2}}}&\lesssim 	\left\|\eta(t)\Pi_{N_{\lambda}}\partial_{x}\Big[\Pi_{>N_{\lambda}}u^*_{\lambda}(t)\cdot\left(u^*_{\lambda}(t)+\Pi_{N_{\lambda}}u^*_{\lambda}(t)\right)\Big]\right\|_{Z^{-\frac{1}{2}}}\notag\\
		&\lesssim \lambda^{0+}\left\|\Pi_{>N_{\lambda}}u^*_{\lambda}(t)\right\|_{Y^{-\frac{1}{2}}}\cdot \left(\|{u^*_{\lambda}}\|_{Y^{-\frac{1}{2}}}+\left\|\Pi_{N_{\lambda}}{u^*_{\lambda}}\right\|_{Y^{-\frac{1}{2}}}\right)\notag\\
		&\lesssim \lambda^{0+}N_{\lambda}^{-\frac{1}{2}-s}\|{u^*_{\lambda}}\|_{Y^s}^2.
	\end{align}
	Finally,  substituting \eqref{eq:estimate_B1B3}, \eqref{eq:R_11} and \eqref{eq:R_12} into \eqref{eq:A_3} and noting that $\|{u^*_{\lambda}}\|_{Y^s}\leq C^{\#}\|\phi_{\lambda}\|_{H^s}\leq C^{\#}\lambda^{-\frac{3}{2}-s}\|\phi\|_{H^s}$, we obtain the desired estimate \eqref{eq:estimate_R1}.
\end{proof}

\subsection{The proof of Theorem~\ref{thm:main} for short time}\label{subsec:proof_small_T}
From the discussions in Remark~\ref{lem:relation_Gamma}, equation \eqref{eq:u-u*} and the rescaling relationship \eqref{eq:relation_scaling}, we deduce that
\begin{align}\label{eq:final_ineq_1}
	\max_{\sigma_n \leq \lambda^{-3}} \|u(\sigma_n) - u^n\|_{H^{-\frac{1}{2}}}
	&\leq \lambda^{\frac{3}{2}} \max_{t_n \leq 1} \|u_{\lambda}(t_n) - u_{\lambda}^n\|_{H^{-\frac{1}{2}}}\notag\\
	&\leq \lambda^{\frac{3}{2}}\|u_{\lambda}-\U\|_{L^{\infty}(0,1;H^{-\frac{1}{2}})}\lesssim \lambda^{\frac{3}{2}}\|e\|_{Y^{-\frac{1}{2}}},
\end{align}
where the error function $e=u^*_\lambda-\U^*$ is defined in \eqref{eq:def_error}. The scaling parameter $\lambda\geq1$ satisfies the conditions in Theorem~\ref{thm:existence} and \eqref{eq:condition_lambda} below. Importantly, $\lambda$ is independent of $N$ and $\tau$ but possibly depends on $\|\phi\|_{H^s}$. Consequently, if $T< \lambda^{-3}$ the proof of the error estimate in \eqref{eq:main_thm}  reduces to estimating $\|e\|_{Y^{-\frac{1}{2}}}$.

By combining \eqref{eq:Y_norm_inequality}, \eqref{eq:e_0_Y12}, \eqref{eq:e_times_u_Y12}, and the estimates from Lemmas~\ref{lem:estimate_R1} and \ref{lem:estimate_A1A2}, we obtain the following inequality:
\begin{align}\label{eq:e_estimate}
	\|e\|_{Y^{-\frac{1}{2}}}\leq& C\lambda^{0+}\left(\left\|u^{*}_{\lambda}\right\|_{Y^{-\frac{1}{2}}}+\left\|\U^{*}\right\|_{Y^{-\frac{1}{2}}}\right)\cdot\|e\|_{Y^{-\frac{1}{2}}}+C \lambda^{-\frac{5}{2}-s+} N^{-\frac{1}{2}-s}+C\lambda^{-\frac{3}{2}}\tau N^{3(\frac{1}{2}-s)},
\end{align}
where $C$ is a positive constant independent of $N$, $\tau$ and $\lambda$ but possibly dependent on $\|\phi\|_{H^s}$. Since $$\left\|u^{*}_{\lambda}\right\|_{Y^{-\frac{1}{2}}}+\left\|\U^{*}\right\|_{Y^{-\frac{1}{2}}}\leq2C^{\#}\|\phi_{\lambda}\|_{H^s}\leq 2C^{\#}\lambda^{-\frac{3}{2}-s}\|\phi\|_{H^s},$$
we can choose a sufficient large $\lambda$ such that 
\begin{align}\label{eq:condition_lambda}
	C\lambda^{0+}\left(\left\|u^{*}_{\lambda}\right\|_{Y^{-\frac{1}{2}}}+\left\|\U^{*}\right\|_{Y^{-\frac{1}{2}}}\right)\leq 2CC^{\#}\lambda^{-\frac{3}{2}-s+}\|\phi\|_{H^s}\leq \frac{1}{2}.
\end{align}
Thus, we obtain the estimate
\begin{align}\label{estimate_e}
	\|e\|_{Y^{-\frac{1}{2}}}\lesssim& \lambda^{-\frac{5}{2}-s+} N^{-\frac{1}{2}-s}+\lambda^{-\frac{3}{2}}\tau N^{3(\frac{1}{2}-s)}.
\end{align} 
By combining  \eqref{eq:final_ineq_1} and \eqref{estimate_e},  taking $N=O(\tau^{-\frac{1}{2-2s}+})$, and noting that $\lambda^{-1-s+}\leq 1$ for $\lambda\geq1$ and $s\in(-\frac{1}{2},0]$,
we obtain
\begin{align}\label{eq:iterate_ineq}
	&\sup_{0\leq n\leq T/\tau}\|u(\sigma_n)-u^n\|_{H^{-\frac{1}{2}}}\lesssim  \lambda^{-1-s+}N^{-\frac{1}{2}-s}+\tau N^{3(\frac{1}{2}-s)}\lesssim 
	\tau^{\frac{1+2s}{4-4s}-}.
\end{align}
This completes the proof of Theorem~\ref{thm:main} for the problem \eqref{eq:KdV_equation} within short time interval, i.e. $T<\lambda^{-3}$ for $\lambda$ satisfying the conditions in Theorem~\ref{thm:existence} and \eqref{eq:condition_lambda}.

\subsection{The proof of Theorem~\ref{thm:main} for large time}

For the case $T \ge \lambda^{-3}$, since it has already been established in Subsection~\ref{subsec:proof_small_T} that the conclusion of Theorem~\ref{thm:main} holds for $0\leq n \leq M_0$ where $M_0:=\lfloor\frac{\lambda^{-3}}{\tau}\rfloor$, we can proceed from $u(\sigma_{M_0})$ and $u^{M_0}$ as initial data and then repeatedly apply similar arguments to prove the error estimates for $n > M_0$. Specifically, we first consider the following nonlinear operators:
\begin{align}\label{eq:def_Gamma_1}
	\Gamma_1(w)(t) = & \, \eta(t)e^{-t \partial_x^3}u^{M_0}_{\lambda} + \frac{1}{2}\eta(t)\int_{0}^{t}e^{(s-t)\partial_x^3}\eta(s)\partial_{x}\left(w(s)\right)^2 \, \mathrm{d}s + \mathcal{A}(w)(t),
\end{align}
and
\begin{align}\label{eq:def_tilde_Gamma_1}
	\widetilde{\Gamma}_1(w)(t)=\eta(t)e^{-t \partial_x^3}u_{\lambda}(t_{M_0})+\frac{1}{2}\eta(t)\int_{0}^{t}e^{(s-t)\partial_x^3}\eta(s)\partial_{x}\left(w(s)\right)^2\d s,
\end{align}
where $\mathcal{A}$ in \eqref{eq:def_Gamma_1} is defined as in \eqref{eq:def_Gamma}.  And then convert the estimate $\|u(\sigma_n)-u^n\|_{H^{-\frac{1}{2}}}$ for $M_0\leq n\leq 2M_0$ into the Bourgain norm estimate between the two fix points of \eqref{eq:def_Gamma_1} and \eqref{eq:def_tilde_Gamma_1}, as the discussions in \eqref{eq:final_ineq_1}.

It is important to observe from Theorem~\ref{thm:existence} and \eqref{eq:condition_lambda} that the choice of the scaling parameter $\lambda$ depends on the $H^s$ norm of the initial data, specifically $\|u^{M_0}_{\lambda}\|_{H^s}$. This implies that for large $T$, ensuring the stability of the numerical solution in the $H^s$ space becomes crucial. However, directly applying the results from Theorem~\ref{thm:existence} is quite challenging. The difficulty arises because Theorem~\ref{thm:existence} only provides the following boundedness estimate:  
\begin{align}\label{eq:boundedness_estimate}  
	\|u^{M_0}_{\lambda}\|_{H^s} \leq C^{\#}\|\phi_{\lambda}\|_{H^s},  
\end{align}  
for some constant $C^{\#} > 1$. When $T$ becomes significantly larger compared to $\lambda^{-3}$, it is necessary to partition the entire time interval into smaller segments and iteratively apply short-time error estimates within each segment.  
Relying only on the boundedness estimate \eqref{eq:boundedness_estimate} leads to exponential growth in the $H^s$-norm estimate of the numerical solution. Consequently, $\lambda$ must be chosen sufficiently large, which in turn causes the length of the time segments, $\lambda^{-3}$, to become extremely small. This makes it challenging to extend the convergence results in Subsection~\ref{subsec:proof_small_T} to longer time intervals.

One approach to addressing this difficulty is to incorporate the error estimate and Bernstein's inequality by choosing  
$N \sim \tau^{-\frac{1}{2-2s}+\varepsilon}$  
for some small $\varepsilon > 0$ with $\varepsilon \ll 1$. This choice introduces an additional small term in the error estimate, allowing us to establish a uniform bound on the numerical solution in $H^{\gamma}$ for a fixed $\gamma \in (-1/2, s)$ as $\tau$ tends to zero (see, for example, \eqref{eq:error_H_gamma_2} below). 

From Theorem~\ref{thm:existence}, we observe that the nonlinear operator \eqref{eq:def_Gamma_1} admits a unique fixed point in the Bourgain space $Y^{\gamma}$ under the choice of the scaling parameter $\lambda$, which depends only on the $H^{\gamma}$ norm of the numerical solution. Consequently, for the entire time interval $[0,T]$, we can select a uniform time segment length,  
$\lambda^{-3}$,  
which depends solely on $\|u\|_{L^\infty(0,T; H^s)}$. By iteratively applying the short-time error estimate from Subsection~\ref{subsec:proof_small_T}, we obtain a convergence result for the long-time interval.  

Meanwhile, from Lemma~\ref{lem:estimate_A1A2}, we note that part of the consistency error estimate of our numerical method depends on the regularity of the initial value for the nonlinear operator \eqref{eq:def_Gamma_1}, i.e., the regularity of the numerical solution. Therefore, to ensure the validity of the error estimate, we choose $\gamma < s$ such that  
$|s - \gamma|$  
is sufficiently small and satisfies:

\begin{align}\label{eq:gamma_condition_2}
	\frac{3}{2-2s}(\gamma-s)+(2-2\gamma)\varepsilon>0,\quad \text{for}\quad N=O(\tau^{-\frac{1}{2-2s}+\varepsilon}), \quad \text{and}\quad s\in(-\frac{1}{2},0],
\end{align}
where $\varepsilon$ is fixed for all time-steps. And we will prove that there exist a $\tau_0$ sufficiently small such that if $\tau<\tau_0$ there holds
\begin{subequations}\label{reformulated_problem}
	\begin{align}
		\|u(\sigma_n)-u^n\|_{H^{-\frac{1}{2}}}&\leq C \tau^{\frac{1+2s}{4-4s}-\frac{1+2s}{2}\varepsilon},\label{reformulate_1}\\
		\|u^n\|_{H^{\gamma}}&\leq \|u\|_{L^{\infty}(0,T;H^{s})}+1,\label{reformulate_2}
	\end{align}
\end{subequations}
for all $0\leq n\leq T/\tau$, where $C$ is a positive constant independent of $\tau$ and $N$ but possibly depends on $T$ and $\|u\|_{L^{\infty}(0,T;H^{s})}$.

Firstly, we infer from \eqref{eq:iterate_ineq} that
\begin{align}\label{eq:reformulated_estimate}
	\sup_{0\leq n\leq M_0}\|u(\sigma_n)-u^n\|_{H^{-\frac{1}{2}}}&\lesssim  \lambda^{-1-s+}N^{-\frac{1}{2}-s}+\tau N^{3(\frac{1}{2}-s)}\notag\\
	&\lesssim \lambda^{-1-s+}\tau^{\frac{1+2s}{4-4s}-\frac{1+2s}{2}\varepsilon}+\tau ^{1+3(\frac{1}{2}-s)(-\frac{1}{2-2s}+\varepsilon)}\notag\\
	&=\tau^{\frac{1+2s}{4-4s}-\frac{1+2s}{2}\varepsilon}\left(\lambda^{-1-s+}+\tau^{(2-2s)\varepsilon}\right)\lesssim \tau^{\frac{1+2s}{4-4s}-\frac{1+2s}{2}\varepsilon}.
\end{align}
Thus, we have proved  \eqref{reformulate_1} for $0\leq n\leq M_0$.

Next, from the definition of the filtered low-regularity integrator \eqref{eq:filtered_LRI}, we have
\begin{align*}
	\Pi_{>N}u^{n}=\Pi_{>N}\left(e^{-n \tau\partial_x^3}\phi\right),
\end{align*}
for all  $0\leq n\leq T/\tau$.
Thus, using \eqref{eq:reformulated_estimate} and the Bernstein's inequality, we obtain
\begin{align}\label{eq:H_gamma_estimate}
	\|u^{n}-u(\sigma_{n})\|_{H^\gamma}&\leq 	\left\|\Pi_{N}\left(u^{n}-u(\sigma_{n})\right)\right\|_{H^\gamma}+\|\Pi_{>N}u^{n}\|_{H^\gamma}+\|\Pi_{>N}u(\sigma_{n})\|_{H^{\gamma}}\notag\\
	&\lesssim \tau^{\frac{1+2s}{4-4s}-\frac{1+2s}{2}\varepsilon}N^{\gamma+\frac{1}{2}}
	+N^{\gamma-s}\left(\|\phi\|_{H^s}+\|u\|_{L^\infty(0,T; H^s)}\right),
\end{align}
for all $0\leq n\leq M_0$.
Through direct calculation, we deduce the following result
\begin{align}\label{eq:error_H_gamma_2}
	\|u^{n}-u(\sigma_{n})\|_{H^\gamma}\lesssim\tau^{(s-\gamma)(\frac{1}{2-2s}-\varepsilon)}.
\end{align}
Since $\gamma<s$,
the exponent of $\tau$ in \eqref{eq:error_H_gamma_2} is strictly positive.
There exists a constant $\tau_0>0$ such that, if $\tau<\tau_0$, then
\begin{align*}
	\|u^{n}-u(\sigma_{n})\|_{H^\gamma}\leq 1.
\end{align*}
Therefore, we have
\begin{align}\label{eq:un_H_gamma}
	\|u^{n}\|_{H^\gamma}\leq \|u(\sigma_{n})\|_{H^{\gamma}}+1\leq \|u\|_{L^{\infty}(0,T;H^{s})}+1,
\end{align}
which is the uniform estimate of the numerical solution in $H^{\gamma}$, i.e. \eqref{reformulate_2} for all $0\leq n\leq M_0$.

Now, we starting from $u(\sigma_{M_0})$ and $u^{M_0}$ as initial data, from Theorem~\ref{thm:existence}, we have, there exists a positive constant $\lambda_0$, which is independent of $N$ and $\tau$ but may dependent on $\|u^{M_0}\|_{H^\gamma}$, such that for any $\lambda>\lambda_0$ the equation $\Gamma_1(\U^{*,1})=\U^{*,1}$ admits a solution $\U^{*,1}\in Y^{\gamma}$ satisfying
\begin{align}
	\|\U^{*,1}\|_{Y^{\gamma}}\leq C^{\#}\|u^{M_0}_{\lambda}\|_{H^\gamma}\leq C^{\#} \lambda^{-\frac{3}{2}-\gamma}\|u^{M_0}\|_{H^\gamma}.
\end{align}
Moreover, we note from the definition of $\Gamma_1$ in \eqref{eq:def_Gamma_1} that for $t\in[0,1]$ there holds
\begin{align}
	\U^{*,1}(t)\equiv \U(t_{M_0}+t).
\end{align}

\begin{remark}
	\upshape
	Since $\lambda_0$ depends only on $\|u^{M_0}\|_{H^\gamma}$ which is bounded by the fixed constant $\|u\|_{L^{\infty}(0,T;H^{s})}+1$ and is independent of $\tau$ and the numerical solution $u^n$,  we can, without loss of generality, assume that $\lambda_0$ here is the same as previously defined. 
	This ensures consistent use of the same scaling constant $\lambda$ throughout the analysis.
	Meanwhile, we note that, the conditions $\gamma<s$ and \eqref{eq:gamma_condition_2} are sufficient to guarantee the condition
	$N\lesssim \tau^{-\frac{1}{2-2\gamma}}$, which is necessary for maintaining the stability of the numerical scheme with $H^{\gamma}$ initial data (see Theorem~\ref{thm:existence}).
	To see this, note that for $N=O(\tau^{-\frac{1}{2-2s}+\varepsilon})$, the condition above is equivalent to
	\begin{align*}
		-\frac{1}{2-2\gamma}\leq -\frac{1}{2-2s}+\varepsilon,
	\end{align*}
	which simplifies to $	\frac{1}{2}(s-\gamma)+\frac{3}{2-2s}(\gamma-s)+(2-2\gamma)\varepsilon\geq 0.$
	This inequality follows directly from the conditions $\gamma<s$ and \eqref{eq:gamma_condition_2}.
\end{remark}

We note from \cite[Proposition~6]{Colliander} that, since $u_{\lambda}(t_{M_0})\in H^s$, the equation 
\begin{align}
	\widetilde{\Gamma}(u^{*,1}_{\lambda})=u^{*,1}_{\lambda},
\end{align}
admits a unique solution in the $Y^s$ space, satisfying
$$\|u_{\lambda}^{*,1}\|_{Y^s}\leq C^{\#}\|u_{\lambda}(t_{M_0})\|_{H^s},$$
and
\begin{align}\label{eq:gamma_1_relation_u*_u}
	u_{\lambda}^*(t)=u_{\lambda}(t_{M_0}+t),\quad \text{for}\quad t \in [0, 1].
\end{align}
Thus,  the error between $\U(t)$ and $u_{\lambda}(t)$ for $t\in [M_0\tau_{\lambda},M_0\tau_{\lambda}+1]$ can be estimated by analyzing $u^{*,1}_{\lambda}-\U^{*,1}$ in the Bourgain space $Y^{-\frac{1}{2}}$. Denoting
\begin{align*}
	e_1(t,x):=u^{*,1}_{\lambda}-\U^{*,1},
\end{align*}
and following the same arguments as in \eqref{eq:Y_norm_inequality}--\eqref{eq:e_times_u_Y12}, we deduce
\begin{align}\label{eq:gamma_Y_norm_inequality}
	\|e_1\|_{Y^{-\frac{1}{2}}}\lesssim&\left\|e_1(0)\right\|_{H^{-\frac{1}{2}}}+\lambda^{0+}\left(\left\|u^{*,1}_{\lambda}\right\|_{Y^{-\frac{1}{2}}}+\left\|\U^{*,1}\right\|_{Y^{-\frac{1}{2}}}\right)\cdot\|e_1\|_{Y^{-\frac{1}{2}}}\notag\\
	&\quad +\|\mathcal{A}_1(\U^{*,1})\|_{Y^{-\frac{1}{2}}}+\|\mathcal{A}_2(\U^{*,1})\|_{Y^{-\frac{1}{2}}}+\|\mathcal{A}_3(\U^{*,1})\|_{Y^{-\frac{1}{2}}}.
\end{align}
Since 	$\|u^{M_0}\|_{H^\gamma}\leq \|u\|_{L^{\infty}(0,T;H^{s})}+1$, it follows from Lemma~\ref{lem:estimate_A1A2} that
\begin{align}\label{eq:gamma_estimate_A1A2}
	\|\mathcal{A}_1(\U^{*,1})\|_{Y^{-\frac{1}{2}}}+\|\mathcal{A}_2(\U^{*,1})\|_{Y^{-\frac{1}{2}}}\leq C\lambda^{-\frac{3}{2}}\tau N^{3(\frac{1}{2}-\gamma)},
\end{align}
where $C$ is a positive constant independent of $N$, $\tau$ and $\lambda$ but possibly dependent on $	\|u^{M_0}_{\lambda}\|_{H^\gamma}$.

From Lemma~\ref{lem:estimate_R1}, the estimate of $\mathcal{A}_3(\U^{*,1})$ can be decomposed into two parts: the projection error for the nonlinear term of $u^{*,1}_{\lambda}$ and the difference between the nonlinear terms of $u^{*,1}_{\lambda}$ and $\U^{*,1}$. Since $u^{*,1}_{\lambda} \in Y^s$, we obtain the following estimate:
\begin{align}\label{eq:gamma_estimate_R1}
	\|\mathcal{A}_3(\U^{*,1})\|_{Y^{-\frac{1}{2}}}\leq C\left[\lambda^{-\frac{5}{2}-s+}N^{-\frac{1}{2}-s}+\lambda^{0+}\left(\left\|u^{*,1}_{\lambda}\right\|_{Y^{-\frac{1}{2}}}+\left\|\U^{*,1}\right\|_{Y^{-\frac{1}{2}}}\right)\cdot\|e_1\|_{Y^{-\frac{1}{2}}}\right].
\end{align}

Combining \eqref{eq:gamma_estimate_A1A2}, \eqref{eq:gamma_Y_norm_inequality}, \eqref{eq:gamma_estimate_R1}, and  applying the same reasoning as in \eqref{eq:e_estimate}--\eqref{eq:iterate_ineq}, we obtain
\begin{align}\label{eq:iterate_ineq_new}
	&\sup_{M_0\leq n\leq 2M_0}\|u(\sigma_n)-u^n\|_{H^{-\frac{1}{2}}}\notag\\
	&\quad\leq C_0 \|u(\sigma_{M_0})-u^{M_0}\|_{H^{-\frac{1}{2}}}+C_0\lambda^{-1-s+} N^{-\frac{1}{2}-s}+C_0\tau N^{3(\frac{1}{2}-\gamma)}\notag\\
	&\quad\leq C_0\left(\left\|u(\sigma_{M_0})-u^{M_0}\right\|_{H^{-\frac{1}{2}}}+\tau^{\frac{1+2s}{4-4s}-\frac{1+2s}{2}\varepsilon}\right),
\end{align}
with constant $C_0$ independent of $\tau$, $N$ and $\lambda$. Here we have used the relation $N=O(\tau^{-\frac{1}{2-2s}+\varepsilon})$, and 
\begin{align*}
	\tau N^{3(\frac{1}{2}-\gamma)}\sim\tau ^{1+3(\frac{1}{2}-\gamma)(-\frac{1}{2-2s}+\varepsilon)}=\tau^{\frac{1+2s}{4-4s}-\frac{1+2s}{2}\varepsilon}\cdot \tau^{\frac{3}{2-2s}(\gamma-s)+(2-2\gamma)\varepsilon+(s-\gamma)\varepsilon}
\end{align*}
where
\begin{align*}
	\frac{3}{2-2s}(\gamma-s)+(2-2\gamma)\varepsilon+(s-\gamma)\varepsilon>0
\end{align*}
by utilizing $s>\gamma$ and the condition \eqref{eq:gamma_condition_2}. 
By further combining \eqref{eq:reformulated_estimate} and \eqref{eq:iterate_ineq_new}, we establish the error estimate \eqref{reformulate_1} for all $M_0 \leq n \leq 2M_0$. Then, following the same steps as in \eqref{eq:H_gamma_estimate}--\eqref{eq:un_H_gamma}, we obtain the uniform $H^{\gamma}$ stability of the numerical method for all $M_0 \leq n \leq 2M_0$:
\begin{align}\label{eq:un_H_gamma_new}
	\sup_{M_0\leq n\leq 2M_0}\|u(\sigma_n)-u^n\|_{H^{\gamma}}\leq \|u\|_{L^{\infty}(0,T;H^{s})}+1,
\end{align}
for $\tau$ sufficiently small.

Since the time interval $[0,T]$ can be divided into at most $L:=\lfloor \frac{T}{M_0\tau}\rfloor+1$ slices, and
\begin{align*}
	\left\lfloor \frac{T}{M_0\tau}\right\rfloor+1\lesssim T\lambda^3,
\end{align*}
which is independent of $\tau$ and $N$, but may depends only on $T$ and $\|u\|_{L^{\infty}(0,T;H^{s})}$. 
We can iterate \eqref{eq:iterate_ineq_new} and \eqref{eq:un_H_gamma_new} in the classical manner to derive 
\begin{align}
	\sup_{0\leq n\leq T/\tau}\|u(\sigma_n)-u^n\|_{H^{-\frac{1}{2}}}\leq (1+C_0)^{L}\cdot\tau^{\frac{1+2s}{4-4s}-\frac{1+2s}{2}\varepsilon},
\end{align}
and
\begin{align}
	\sup_{0\leq n\leq T/\tau}\|u^n\|_{H^{\gamma}}\leq \|u\|_{L^{\infty}(0,T;H^{s})}+1
\end{align}
for $\tau$ sufficiently small. From this, it is clear that the conclusion \eqref{eq:main_thm} has been established, thereby completing the proof of Theorem~\ref{thm:main}.

\section{Numerical experiments}\label{sec:numerical_experiments}
In this section, we conduct numerical experiments to substantiate the theoretical analysis presented earlier. We demonstrate the convergence properties of the numerical solutions obtained using the proposed filtered low-regularity integrator, specifically for computing nonsmooth solutions of the Korteweg-de Vries (KdV) equation in spaces below $L^2$. We consider the KdV equation on the torus $[-\pi, \pi]$ with the following initial condition:
\begin{equation}
	u^{0}(x) = \frac{1}{10} \sum_{0 \neq k \in \mathbb{Z}} q_k|k|^{-0.501-s} e^{ikx},
\end{equation}
where $q_k\in\mathbb{C}$ is a phase factor satisfying $|q_k|=O(1)$, which ensures that the initial function $u^0$ belongs to $H^{s}$ but not to $H^{s+0.001}$. For $s=0$ and $s=-\frac{1}{4}$, we present the graph of the function $u^0$ in Figures~\ref{fig:6-0}(a) and (b), highlighting the significant singularities at $x=0$.

To address the KdV equation with nonsmooth initial data, specifically when $u_0 \in L^2$ and $u_0 \in H^{-\frac{1}{4}}$, we utilize the filtered low-regularity integrator, given by:
\begin{itemize}
	\item $\tau = 2^{-16} T$ and $N = O(\tau^{-\frac{1}{2}}) = 2^{-8}$ for $s = 0$,
	\item $\tau = 2^{-20} T$ and $N = O(\tau^{-\frac{2}{5}}) = 2^{-8}$ for $s = -\frac{1}{4}$.
\end{itemize}
The spatial discretization is performed using a Fourier pseudo-spectral method with $2^{14}$ modes.
\begin{figure}[htb!]
	\centering
	\subfigure{\includegraphics[width=6.1cm,height=4.8cm]{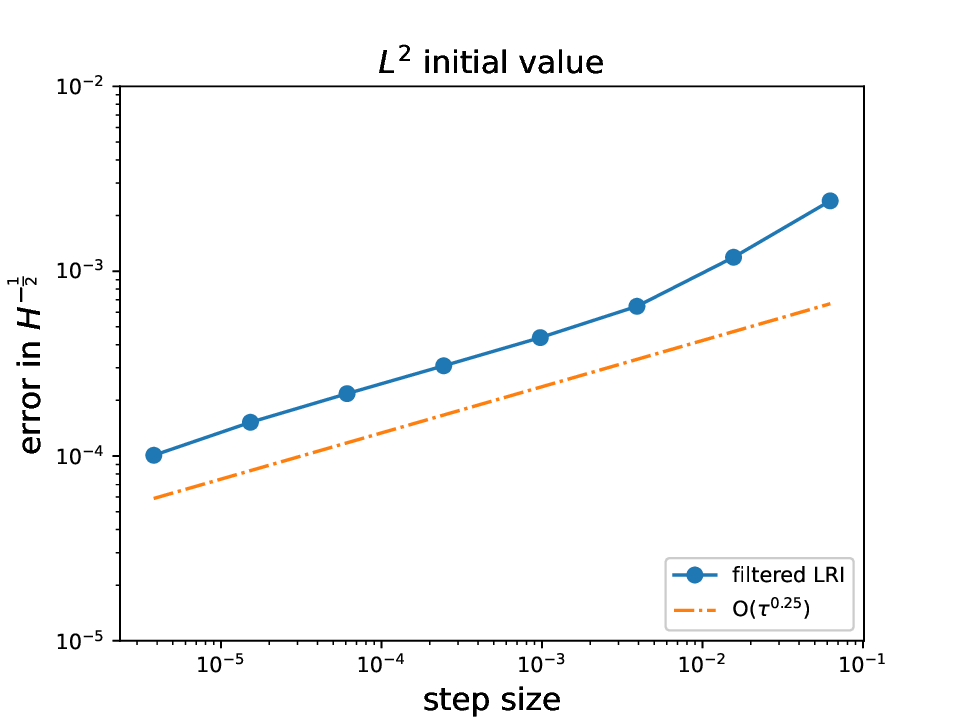}}
	\qquad
	\subfigure{\includegraphics[width=6.1cm,height=4.8cm]{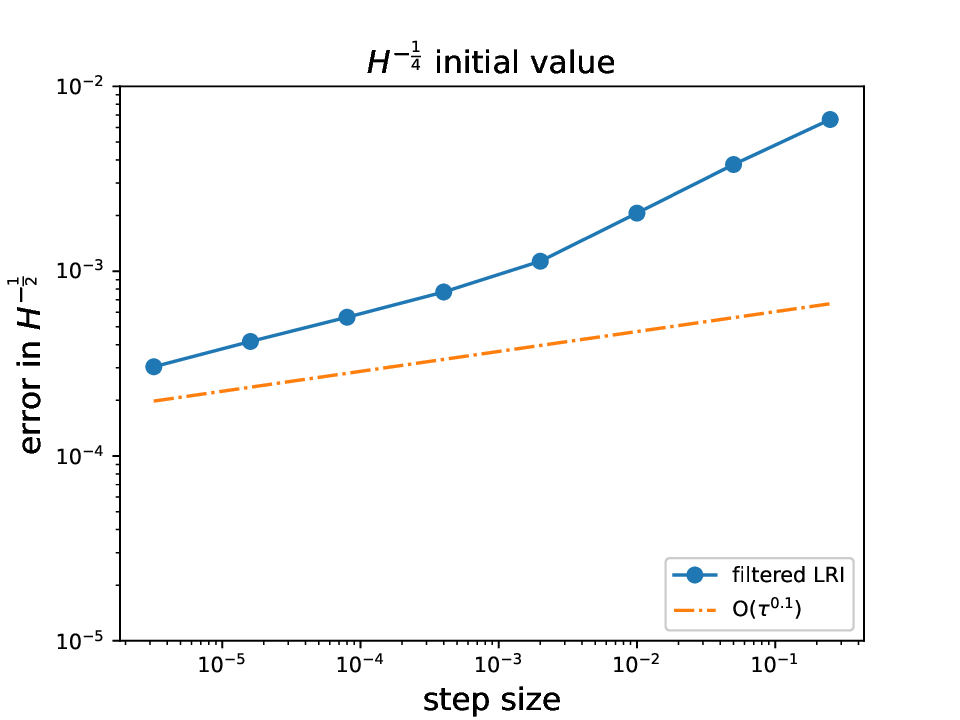}}
	\caption{Errors of the numerical solutions for $L^2$ and $H^{-\frac{1}{4}}$ initial values}
	\label{fig:6-1}
\end{figure}
In Figure~\ref{fig:6-1} we present the errors of the numerical solutions computed by the proposed method in \eqref{eq:filtered_LRI} at $T=1$ for different time steps $\tau$ and $N=\tau^{-\frac{1}{2-2s}}$, with the same spatial discretization using $2^{14}$ Fourier modes. The reference solutions are obtained using a sufficiently small $\tau$, specifically $\tau=2^{-20}$ for $s=0$ and $\tau=2^{-25}$ for $s=-\frac{1}{4}$. The numerical results in Figure \ref{fig:6-1} show that the convergence order of the proposed method for $H^{s}$ initial value is $\frac{1+2s}{4-4s}$ for $s=0$ and $s=-\frac{1}{4}$. This is consistent with the theoretical result proved in Theorem~\ref{thm:main}.  
\begin{figure}[htb!]
	\centering
	\subfigure[$L^2$ initial value]{\includegraphics[width=6.1cm,height=4.8cm]{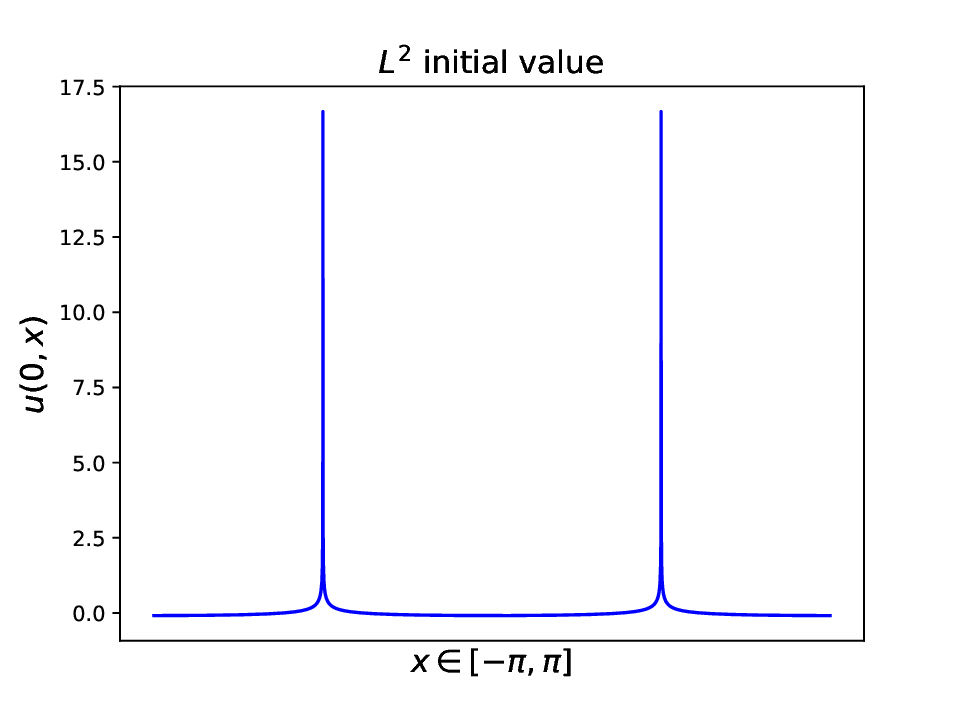}}
	\qquad
	\subfigure[$H^{-\frac{1}{4}}$ initial value]{\includegraphics[width=6.1cm,height=4.8cm]{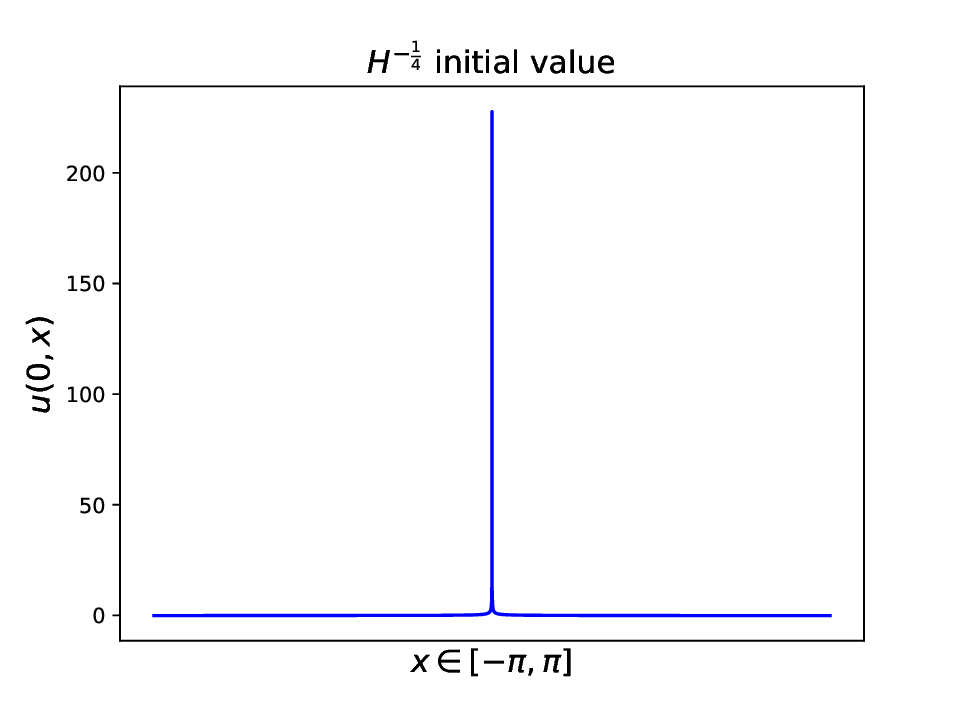}}
	\vfill
	\centering
	\subfigure[Numerical solution at $T=1$ with $L^2$ initial value]{\includegraphics[width=6.1cm,height=4.8cm]{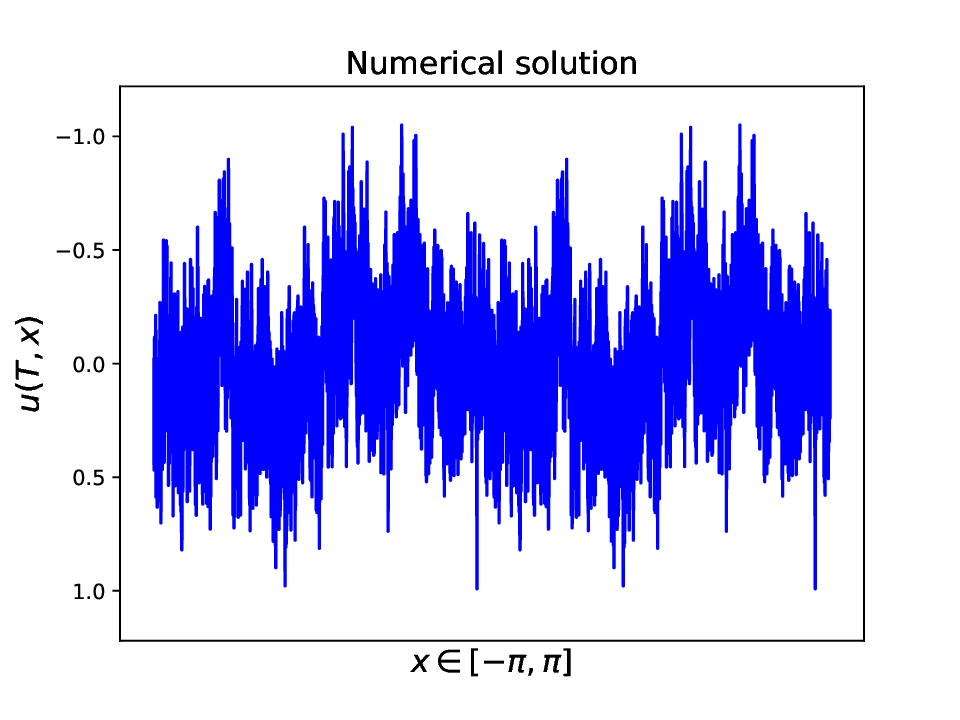}}
	\qquad
	\subfigure[Numerical solution at $T=1$ with $H^{-\frac{1}{4}}$ initial value]{\includegraphics[width=6.1cm,height=4.8cm]{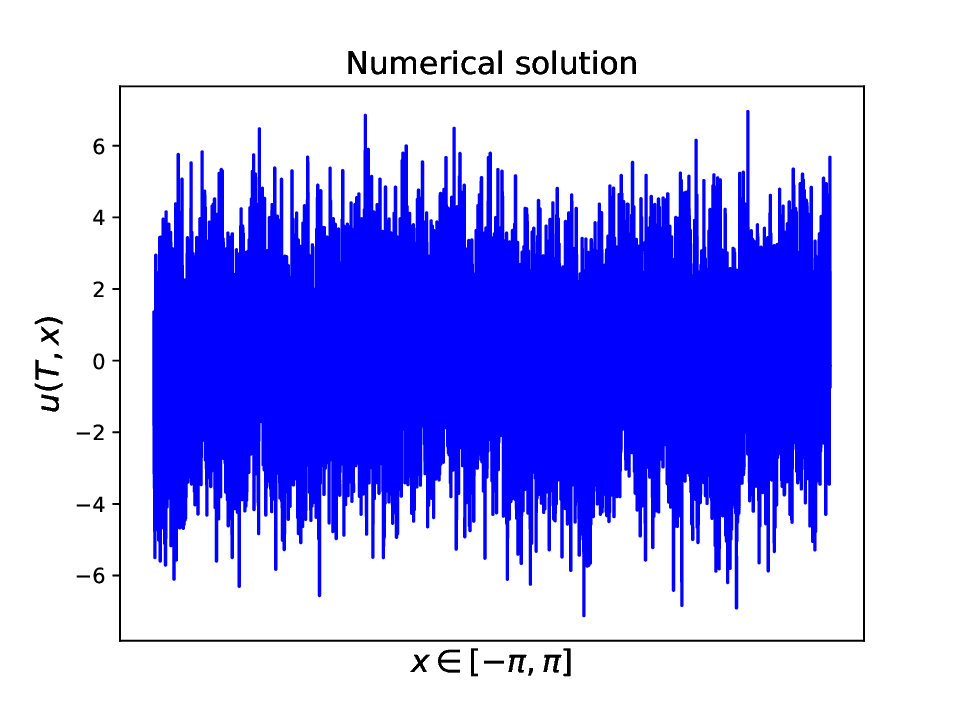}}
	\vfill
	\centering
	\subfigure[Numerical solution at $T=\pi/32$ with $L^2$ initial value]{\includegraphics[width=6.1cm,height=4.8cm]{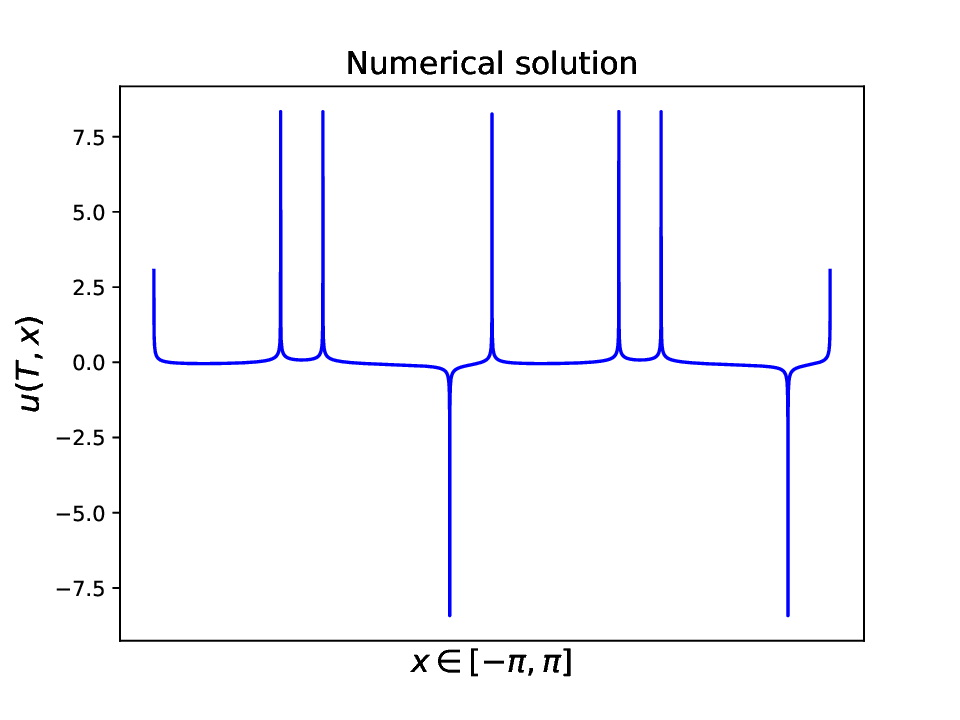}}
	\qquad
	\subfigure[Numerical solution at $T=\pi/32$ with $H^{-\frac{1}{4}}$ initial value]{\includegraphics[width=6.1cm,height=4.8cm]{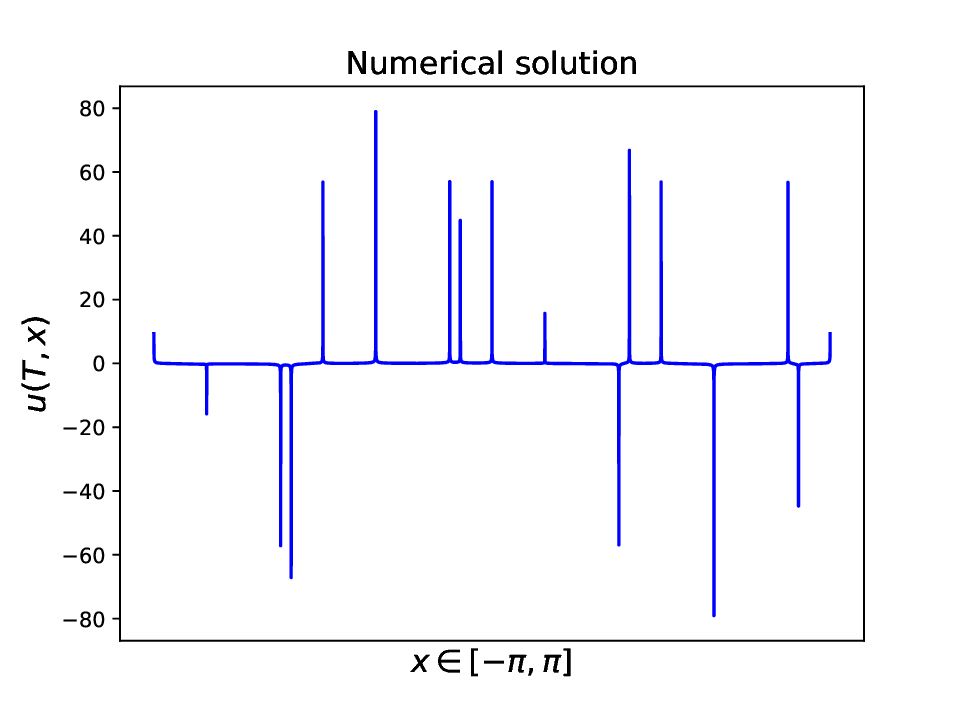}}
	\vspace{5pt}
	\caption{Numerical simulations for $L^2$ and $H^{-\frac14}$ initial values}
	\label{fig:6-0}
\end{figure}
The numerical solutions for initial data in $L^2$ with $q_k = \frac{1}{2} \left( e^{\frac{ik\pi}{2}} + e^{-\frac{ik\pi}{2}} \right)$ and in $H^{-\frac{1}{4}}$ with $q_k = 1$ are shown in Figure~\ref{fig:6-0}(c)--(f) for final times $T = 1$ and $T = \frac{\pi}{32}$, respectively. The accuracy of these results is corroborated by the convergence tests shown in Figure~\ref{fig:6-1}. These findings illustrate markedly different behaviors of the solutions depending on whether the final time $T$ is a rational or irrational multiple of the interval length. Specifically, when $T = 1$, an irrational multiple of the interval length, the initial singularities of the KdV equation lead to pronounced spatial oscillations. Conversely, when $T = \frac{\pi}{32}$, a rational multiple of the interval length, there is a refocusing of the singularities that closely resembles the initial conditions, i.e., the solution has finite concentrations of singularities without spatial oscillations. The numerical results are consistent with the well-known Talbot effect, which finds applications in various fields such as image processing, photolithography, and nonlinear optics \cite{Talbot,Zhang}. These results demonstrate that our numerical algorithm effectively captures the dynamics of these rough solutions.

\section{Conclusion}

We have introduced a novel numerical-analytic framework for solving the KdV equation with initial data in the negative Sobolev spaces, a regime where classical numerical methods typically fail due to their strict regularity requirements. By employing a continuous reformulation strategy, we have transformed the discrete-in-time numerical solution into a continuous-in-time function that satisfies a perturbed KdV equation at the continuous level. This reformulation allows us to convert the stability analysis of the numerical scheme into stability analysis at the continuous level. Consequently, we successfully leveraged the powerful Bourgain-space estimates at the continuous level, combined with suitable rescaling arguments that converts the problem to small initial value and the analysis of the perturbation terms in the Bourgain spaces. This new approach establishes a direct and rigorous bridge between numerical analysis and well-posedness theory based on the Bourgain-space techniques.

With this framework, we have proved convergence of the proposed filtered low-regularity integrator in the $H^{-\frac{1}{2}}$ norm with an order close to ${(1+2s)}/{(4-4s)}$ (with respect to the time stepsize) for initial data in $H^s$, where $-\frac{1}{2}< s \leq 0$, and attains nearly optimal-order convergence with respect to the spatial degrees of freedom (Corollary~\ref{rem:fully_discrete}). Numerical experiments confirm these theoretical findings, demonstrating the robustness and efficiency of our method in accurately capturing subtle dispersive phenomena, such as the Talbot effect, and effectively computing physically realistic, rough initial data.

The analytical and numerical framework introduced here, which combines Bourgain-space techniques with continuous reformulation of the numerical scheme, is not limited to the KdV equation and can be readily adapted to other nonlinear dispersive partial differential equations with solutions of low regularity, thus providing a powerful new approach for future numerical analysis in this challenging regime.

\newpage
\pagestyle{empty}

\appendix

\begin{center}
	\textbf{\Large{Supplementary material}}
\end{center}

\

In this supplementary material, we provide the proofs of some preliminary results presented in the  paper, as well as the detailed proof of the boundedness of the iterative sequence \eqref{eq:iterative_sequence} in the Bourgain space.

\section{Proofs of some results in Section~\ref{sec:bourgain}}\label{Appendix:A}
\begin{proof}[Proof of Lemma~\ref{lem:embedding_Ys}]
	The proof follows directly from the inequality $$\|w\|_{L_{t}^{\infty}H_{x}^s}\leq \|\langle k \rangle^s \tilde{w}(\sigma,k)\|_{L^2((\mathrm{d}k)_{\lambda}) L^1(\mathrm{d}\sigma)}.$$
	
\end{proof}

\begin{proof}[Proof of Lemma~\ref{lem:time_localization}]
	If we use Fourier expansion $\eta(t)=\int_{\mathbb{R}}\hat{\eta}(\sigma_0)e^{it\sigma_0}\d \sigma_0$, then following the same reasoning as in Lemma~2.11 of \cite{Tao}, we obtain
	\begin{align*}
		\|\eta(t)w\|_{X_{s,b}}\leq C_0 \left(\int_{\mathbb{R}}|\hat{\eta}(\sigma_0)|\langle\sigma_0\rangle^{|b|}\d\sigma_0\right)\|w\|_{X_{s,b}},
	\end{align*}
	where $C_0$ is a positive constant depending only on $b$. Since $\eta$ is Schwartz, $\hat{\eta}$ is also rapidly decreasing, and thus the desired estimate $\|\eta(t)w\|_{X_{s,b}}\leq C\|w\|_{X_{s,b}}$ holds with a positive constant $C$ independent of $\lambda$. The other two inequalities in \eqref{eq:time_localization} can be proved in exactly the same manner.
\end{proof}

\begin{proof}[Proof of Corollary~\ref{corollary:Z0}]
	From the definition of the Bourgain space $X_{s,b}$ and the embedding result in Lemma~\ref{lem:embedding_theorem},
	we know that $X_{0,\frac{1}{2}}\hookrightarrow X_{0,\frac{1}{3}}\hookrightarrow L^{4}_t L^{4}_x$. 
	By applying the duality argument, we obtain $(L^{4}_t L^{4}_x)^{\prime}\hookrightarrow (X_{0,\frac{1}{2}})^{\prime}.$
	
	Since $(L^{4}_t L^{4}_x)^{\prime}=L^{\frac{4}{3}}_t L^{\frac{4}{3}}_x$ and $(X_{0,\frac{1}{2}})^{\prime}=X_{0,-\frac{1}{2}}$, it follows that $\|w\|_{X_{0,-\frac{1}{2}}}\lesssim \|w\|_{L^{\frac{4}{3}}_t L^{\frac{4}{3}}_x}$.

	Thus, it remains to show that
	\begin{align}
		\left\|\frac{\tilde{w}(\sigma,k)}{\langle\sigma-4\pi^2k^3\rangle}\right\|_{L^2((\d k)_{\lambda})L^1(\d \sigma)}\lesssim \|w\|_{L^{\frac{4}{3}}_t L^{\frac{4}{3}}_x}.
	\end{align}
	By the duality argument, it suffices to prove
	\begin{align}\label{eq:dulaity_L43}
		\sup_{\|\tilde{f}\|_{L^2_k L^\infty_\sigma}=1}\int\int \langle\sigma-4\pi^2k^3\rangle^{-1}\tilde{w}(\sigma,k)\tilde{f}(\sigma,k)\d \sigma(\d k)_{\lambda}\lesssim \|w\|_{L^{\frac{4}{3}}_t L^{\frac{4}{3}}_x}.
	\end{align}
	By using the Parseval equality, we have
	\begin{align*}
		&\int\int \langle\sigma-4\pi^2k^3\rangle^{-1}\tilde{w}(\sigma,k)\tilde{f}(\sigma,k)\d \sigma(\d k)_{\lambda}\\
		&\quad=\int_0^{\infty}\int_0^{\lambda}w(t,x)\mathscr{F}^{-1}\left(\langle\sigma-4\pi^2k^3\rangle^{-1}\tilde{f}(\sigma,k) \right)(t,x)\d t\d x\\
		&\quad \leq  \|w\|_{L^{\frac{4}{3}}_t L^{\frac{4}{3}}_x}\cdot
		\Big\|\mathscr{F}^{-1}\left(\langle\sigma-4\pi^2k^3\rangle^{-1}\tilde{f}(\sigma,k) \right)(t,x)\Big\|_{L^4_t L^4_x},
	\end{align*}
	where $\mathscr{F}^{-1}$ denotes the inverse Fourier transform.  Then, according to the embedding result \eqref{eq:embedding_theorem}, we obtain
	\begin{align*}
		&\left\|\mathscr{F}^{-1}\left(\langle\sigma-4\pi^2k^3\rangle^{-1}\tilde{f}(\sigma,k) \right)(t,x)\right\|_{L^4_t L^4_x}\\
		&\quad\lesssim \left\|\mathscr{F}^{-1}\left(\langle\sigma-4\pi^2k^3\rangle^{-1}\tilde{f}(\sigma,k) \right)(t,x)\right\|_{X_{0,\frac{1}{3}}} =\left\|\langle\sigma-4\pi^2k^3\rangle^{-\frac{2}{3}}\tilde{f}(\sigma,k)\right\|_{L^2_\sigma L^2_k}\\
		&\quad=\left(\int\int \langle\sigma-4\pi^2k^3\rangle^{-\frac{4}{3}}|\tilde{f}(\sigma,k)|^2\d \sigma (\d k)_{\lambda}\right)^{\frac{1}{2}}\\
		&\quad \leq \left(\int \|\tilde{f}(\sigma,k)\|_{L^{\infty}_{\sigma}}^2\cdot\left(\int \langle\sigma-4\pi^2k^3\rangle^{-\frac{4}{3}}\d \sigma\right) (\d k)_{\lambda}\right)^{\frac{1}{2}}\lesssim \|\tilde{f}\|_{L^2((\d k)_{\lambda})L^{\infty}(\d \sigma)}.
	\end{align*}
	Combining these estimates yields \eqref{eq:dulaity_L43}, which completes the proof.
\end{proof}

\begin{proof}[Proof of Lemma~\ref{lem:Bernstein}]
	Inequalities \eqref{eq:Bernstein_1} and \eqref{eq:Bernstein_3} can be directly proven using the definition of the projection operator. Moreover, from the scaling invariance of the Sobolev inequality, we have:
	\begin{align}
		\|\Pi_{N_{\lambda}} u\|_{L^{4}} \lesssim \|\Pi_{N_{\lambda}} u\|_{H^{\frac{1}{4}}}.
	\end{align}
	Finally, by applying \eqref{eq:Bernstein_1}, we deduce \eqref{eq:Bernstein_2}.
\end{proof}

\section{Proof of Lemma~\ref{lem:boundedness_um}}\label{Appendix:B}

\begin{proof}
	Since $\U^0=0$, the estimate in \eqref{eq:boundedness_um} naturally holds for $\ell=0$. Moreover, from the definition of $\Gamma$ in \eqref{eq:def_Gamma}, we have $\U^1=\eta(t)e^{-t \partial_x^3}\phi_{\lambda}$. By Lemma~\ref{lem:free_solution}, we obtain
	\begin{align*}
		\|\U^1\|_{Y^s}=\|\eta(t)e^{-t \partial_x^3}\phi_{\lambda}\|_{Y^s}\leq C_0\|\phi_{\lambda}\|_{H^s},
	\end{align*}
	where the constant $C_0$ independent of $\tau$, $N$ and $\lambda$. 
	
	Next, we assume that \eqref{eq:boundedness_um} holds with $C^{\#} = 2C_0$ for all $\U^j$ with $j \in \{0, \dots, \ell\}$ and $\ell \geq 1$. 
	By taking the $Y^s$ norm on both sides of $\U^{\ell+1} = \Gamma(\U^\ell)$, with $\Gamma$ given by the following equivalent reformulation of \eqref{eq:def_Gamma}:
	\begin{align}\label{eq:def_Gamma_equivalent}
		\Gamma(w)(t)=&\eta(t)e^{-t \partial_x^3}\phi_{\lambda}\notag\\
		&+\frac{1}{2}\eta(t)\int_{0}^{t}e^{(s-t)\partial_x^3}\eta(s)\Pi_{N_{\lambda}}\partial_{x}\left(\Pi_{N_{\lambda}}w(s)\right)^2\d s+\mathcal{A}_1(w)(t)+\mathcal{A}_2(w)(t),
	\end{align}
	we obtain:
	\begin{align}\label{eq:u_mp1}
		\|\U^{\ell+1}\|_{Y^s}&\leq  \|\eta(t)e^{-t \partial_x^3}\phi_{\lambda}\|_{Y^s}+\left\|\frac{1}{2}\eta(t)\int_{0}^{t}e^{(s-t)\partial_x^3}\eta(s)\Pi_{N_{\lambda}}\partial_{x}\left(\Pi_{N_{\lambda}}\U^\ell(s)\right)^2\d s\right\|_{Y^s}\notag\\
		&\quad+\|\mathcal{A}_1(\U^\ell)\|_{Y^s}+\|\mathcal{A}_2(\U^\ell)\|_{Y^s}\notag\\
		&\leq C_0\|\phi_{\lambda}\|_{H^s}+C\left\|\eta(t)\Pi_{N_{\lambda}}\partial_{x}\left(\Pi_{N_{\lambda}}\U^\ell(t)\right)^2\right\|_{Z^s}+\|\mathcal{A}_1(\U^\ell)\|_{Y^s}+\|\mathcal{A}_2(\U^\ell)\|_{Y^s}\notag\\
		&\leq C_0\|\phi_{\lambda}\|_{H^s}+C\lambda^{0+}\left\|\U^\ell\right\|^2_{X_{s,\frac{1}{2}}}+\|\mathcal{A}_1(\U^\ell)\|_{Y^s}+\|\mathcal{A}_2(\U^\ell)\|_{Y^s},
	\end{align}
	where we have used Lemmas~\ref{lem:free_solution}, \ref{lem:energy_estimate}, \ref{lem:bilinear_estimate} and Remark~\ref{rem:bilinear_estimate} in the above inequalities. We shall estimate each term on the right-hand side of \eqref{eq:u_mp1}, i.e. $\|\mathcal{A}_1(\U^\ell)\|_{Y^s}$ and $\|\mathcal{A}_2(\U^\ell)\|_{Y^s}$.
	
	{\sc Estimate of $\|\mathcal{A}_1(\U^\ell)\|_{Y^s}$:} Recalling the definition of $\mathcal{A}_1$ in \eqref{eq:def_Gamma} and using the results in Lemma~\ref{lem:energy_estimate}, Lemma~\ref{lem:embedding_theorem} and Corollary \ref{corollary:Z0}, we get
	\begin{align}\label{eq:estimate_A_1}
		\|\mathcal{A}_1(\U^\ell)\|_{Y^s}&\lesssim \left\|\eta(t)\Pi_{N_{\lambda}}\partial_{x}\Big(\Pi_{N_{\lambda}}\F(t,\U^\ell)\Pi_{N_{\lambda}}\U^\ell(t)\Big)\right\|_{Z^s}\notag\\
		&\lesssim N_{\lambda}^{1+s}\left\|\eta(t)\Big(\Pi_{N_{\lambda}}\F(t,\U^\ell)\Pi_{N_{\lambda}}\U^\ell(t)\Big)\right\|_{Z^0}\notag\\
		&\lesssim N_{\lambda}^{1+s}\left\|\eta(t)\Big(\Pi_{N_{\lambda}}\F(t,\U^\ell)\Pi_{N_{\lambda}}\U^\ell(t)\Big)\right\|_{L^{\frac{4}{3}}_t L^{\frac{4}{3}}_x}\notag\\
		&\lesssim N_{\lambda}^{1+s}\left\|\eta(t)\Pi_{N_{\lambda}}\F(t,\U^\ell)\right\|_{L^2_t L^2_x}\|\Pi_{N_{\lambda}}\U^\ell(t)\|_{L^4_t L^4_x}\notag\\
		&\lesssim N_{\lambda}^{1+s}\left\|\eta(t)\Pi_{N_{\lambda}}\F(t,\U^\ell)\right\|_{L^2_t L^2_x}\|\Pi_{N_{\lambda}}\U^\ell(t)\|_{X_{0,\frac{1}{3}}}.
	\end{align}
	Since $\|\Pi_{N_{\lambda}}\U^\ell(t)\|_{X_{0,\frac{1}{3}}}\leq \|\Pi_{N_{\lambda}}\U^\ell(t)\|_{X_{0,\frac{1}{2}}}\lesssim N_{\lambda}^{-s}\|\U^\ell\|_{X_{s,\frac{1}{2}}}$, we get
	\begin{align}\label{eq:A_1}
		\|\mathcal{A}_1(\U^\ell)\|_{Y^s}\lesssim N_{\lambda}\left\|\eta(t)\Pi_{N_{\lambda}}\F(t,\U^\ell)\right\|_{L^2_t L^2_x}\|\U^\ell\|_{X_{s,\frac{1}{2}}}.
	\end{align}

	Next, we provide a detailed estimate of $\left\|\eta(t)\Pi_{N_{\lambda}}\F(t, \U^\ell)\right\|_{L^2_t L^2_x}$. Recall that $\F$ is defined in terms of time slices. If we were to directly use the boundedness of the $L_t^\infty H_x^{s}$ norm of $\U^\ell$ to estimate the $L^2_x$ norm of $\F$ at each time step, the resulting estimate would be too coarse to guarantee satisfactory stability results.
	Therefore, we instead use the Bourgain norm of $\U^\ell$ to estimate $\left\|\eta(t)\Pi_{N_{\lambda}}\F(t,\U^\ell)\right\|_{L^2_t L^2_x}$ directly.
	To proceed, we recall \eqref{eq:def_F0}–\eqref{eq:def_F} and introduce the following notation:
	\begin{align}\label{defNmn}
		\mathcal{N}_n^\ell(s) := \Pi_{N_{\lambda}}\U^{\ell}(s) - e^{(t_{n}-s)\partial_x^3}\Pi_{N_{\lambda}}\U^{\ell}(t_n).
	\end{align}
	Then, we have the following estimate:
	\begin{align}\label{eq:F_L2L2}
		&\left\|\eta(t)\Pi_{N_{\lambda}}\F(t,\U^\ell)\right\|_{L^2_t L^2_x}^2\notag\\
		&=\frac{1}{4}\sum_{n}\int_{t_n}^{t_{n+1}}
		\int_0^{\lambda}
		\Big|\eta(t)\int_{t_{n}}^t  e^{(s-t)\partial_x^3}\Pi_{N_{\lambda}}
		\partial_{x}\left(e^{(t_{n}-s)\partial_x^3}\Pi_{N_{\lambda}}\U^{\ell}(t_n)\right)^2\d s
		\Big|^2\d x\d t\notag\\
		& \leq \frac{1}{4}\sum_{n}\int_{t_n}^{t_{n+1}} 
		\int_0^{\lambda}
		\Big|\eta(t)\int_{t_{n}}^t e^{(s-t)\partial_x^3}\Pi_{N_{\lambda}}\partial_{x}
		\big(\Pi_{N_{\lambda}}\U^{\ell}(s)\big)^2\d s
		\Big|^2\d x\d t\notag\\
		&\quad+\frac{1}{2}\sum_{n}\int_{t_n}^{t_{n+1}}\int_0^{\lambda}
		\Big|
		\eta(t)\int_{t_{n}}^t e^{(s-t)\partial_x^3}\Pi_{N_{\lambda}}\partial_{x} 
		\big(\mathcal{N}_n^\ell(s) \Pi_{N_{\lambda}}\U^{\ell}(s)\big)\d s
		\Big|^2\d x\d t\notag\\
		&\quad+\frac{1}{4}\sum_{n}\int_{t_n}^{t_{n+1}}\int_0^{\lambda}
		\Big|\eta(t)\int_{t_{n}}^t e^{(s-t)\partial_x^3}\Pi_{N_{\lambda}}\partial_{x} \big(\mathcal{N}_n^\ell(s)\big)^2\d s
		\Big|^2\d x\d t\notag\\
		&=:I_1+I_2+I_3.
	\end{align}
	We begin by estimating $I_1$ using Bernstein's inequality and Hölder's inequality, yielding the following result:
	\begin{align}\label{eq:I_1}
		I_1&\lesssim N_{\lambda}^2\cdot \sum_{n}\int_{t_n}^{t_{n+1}}\int_0^{\lambda}\Big|\eta(t)\int_{t_{n}}^t e^{(s-t)\partial_x^3}\left(\Pi_{N_{\lambda}}\U^{\ell}(s)\right)^2\d s\Big|^2\d x\d t\notag\\
		&\leq N_{\lambda}^2\cdot\sum_{n}\int_{t_n}^{t_{n+1}}(t-t_n)\int_0^{\lambda}\int_{t_{n}}^t \Big|e^{(s-t)\partial_x^3}\left(\Pi_{N_{\lambda}}\U^{\ell}(s)\right)^2\Big|^2\d s\d x\d t\notag\\
		&=N_{\lambda}^2\cdot\sum_{n}\int_{t_n}^{t_{n+1}}(t-t_n)\left\|e^{(s-t)\partial_x^3}\left(\Pi_{N_{\lambda}}\U^{\ell}(s)\right)^2\right\|^2_{L^2(t_n,t;L^2(0,\lambda))}\d t.
	\end{align}
	Since the operator $e^{(s-t)\partial_x^3}$ preserves the $L^2_x$ norm of function, we have 
	\begin{align*}
		&\Big\|e^{(s-t)\partial_x^3}\left(\Pi_{N_{\lambda}}\U^{\ell}(s)\right)^2\Big\|^2_{L^2(t_n,t;L^2(0,\lambda))}=\Big\|\left(\Pi_{N_{\lambda}}\U^{\ell}(s)\right)^2\Big\|^2_{L^2(t_n,t;L^2(0,\lambda))}\\[2mm]
		&\quad=\left\|\Pi_{N_{\lambda}}\U^{\ell}(s)\right\|^4_{L^4(t_n,t;L^4(0,\lambda))}
		\leq \left\|\Pi_{N_{\lambda}}\U^{\ell}(s)\right\|^4_{L^4(t_n,t_{n+1};L^4(0,\lambda))},
	\end{align*}
	for all $t\in[t_n,t_{n+1}]$ and $n=0,1,2,\cdots$.
	This together with \eqref{eq:I_1} and Lemma~\ref{lem:embedding_theorem} yields
	\begin{align}
		I_1&\lesssim N_{\lambda}^2\cdot\sum_{n}\int_{t_n}^{t_{n+1}}(t-t_n)\d t\cdot \left\|\Pi_{N_{\lambda}}\U^{\ell}(s)\right\|^4_{L^4(t_n,t_{n+1};L^4(0,\lambda))}\notag\\
		&\leq N_{\lambda}^2\tau_{\lambda}^2\sum_{n}\left\|\Pi_{N_{\lambda}}\U^{\ell}(s)\right\|^4_{L^4(t_n,t_{n+1};L^4(0,\lambda))}=N_{\lambda}^2\tau_{\lambda}^2\left\|\Pi_{N_{\lambda}}\U^{\ell}(s)\right\|^4_{L^4(\R;L^4(0,\lambda))}\notag\\
		&\lesssim N_{\lambda}^2\tau_{\lambda}^2\left\|\Pi_{N_{\lambda}}\U^{\ell}\right\|^4_{X_{0,\frac{1}{3}}}\lesssim N_{\lambda}^{2-4s}\tau_{\lambda}^2\left\|\U^{\ell}\right\|^4_{X_{s,\frac{1}{2}}}.
	\end{align}
	The estimate for $I_2$ relies on the following rough estimates for any $t^\prime\geq t$ and functions $v, w$,
	\begin{align}\label{eq:rough_estimate}
		&\Big\|\int_{t}^{t^\prime}e^{(s-t)\partial_x^3}\Pi_{N_{\lambda}}\partial_{x}\Big(\Pi_{N_{\lambda}}v(s)\cdot\Pi_{N_{\lambda}}w(s)\Big)\d s\Big\|_{L^2_x}\notag\\
		&\quad \lesssim (t^\prime-t)N_{\lambda}\|\Pi_{N_{\lambda}}v(s)\|_{L^{\infty}(t,t^\prime;L^{4}(0,\lambda))}\|\Pi_{N_{\lambda}}w(s)\|_{L^{\infty}(t,t^\prime;L^4(0,\lambda))}\notag\\
		&\quad\lesssim(t^\prime-t)N_{\lambda}^{\frac{3}{2}}\|\Pi_{N_{\lambda}}v(s)\|_{L^{\infty}(t,t^\prime;L^2(0,\lambda))}\|\Pi_{N_{\lambda}}w(s)\|_{L^{\infty}(t,t^\prime;L^2(0,\lambda))},
	\end{align}
	where we have applied Lemma~\ref{lem:Bernstein} in the above estimates. 
	By using \eqref{eq:rough_estimate}, we obtain
	\begin{align}\label{eq:estimate_I2}
		I_2&\lesssim \sum_{n}\int_{t_n}^{t_{n+1}}\tau_{\lambda}^2 N_{\lambda}^{3}
		\left\| \mathcal{N}_n^\ell(t) \right\|^2_{L^{\infty}(t_n,t_{n+1};L^2(0,\lambda))}
		\cdot \|\Pi_{N_{\lambda}}\U^\ell(t)\|^2_{L^{\infty}(t_n,t_{n+1};L^2(0,\lambda))}\d t\notag\\
		&\lesssim \tau_{\lambda}^2 N_{\lambda}^{3}\sup_{n}
		\left\|\mathcal{N}_n^\ell(t) \right\|^2_{L^{\infty}_t(t_n,t_{n+1};L^2(0,\lambda))}
		\|\Pi_{N_{\lambda}}\U^\ell(t)\|^2_{L^{\infty}_t L^2_x}.
	\end{align}
	To estimate $\mathcal{N}_n^\ell(t)$, we recall the expression $\U^\ell(t_n)=\Gamma(\U^{\ell-1})(t_n)$. 
	Thus, it follows that
	\begin{align*}
		e^{(t_{n}-t)\partial_x^3}\U^\ell(t_n)=\ &\eta(t_n)e^{-t \partial_x^3}\phi_{\lambda}+\frac{1}{2}\eta(t_n)\int_{0}^{t_n}e^{(s-t)\partial_x^3}\eta(s)\Pi_{N_{\lambda}}\partial_{x}\left(\Pi_{N_{\lambda}}\U^{\ell-1}(s)\right)^2\d s\notag\\
		&-\eta(t_n)\int_{0}^{t_n}e^{(s-t)\partial_x^3}\eta(s)\Pi_{N_{\lambda}}\partial_{x}\Big(\Pi_{N_{\lambda}}\F(s,\U^{\ell-1})\Pi_{N_{\lambda}}\U^{\ell-1}(s)\Big)\d s\\
		&+\frac{1}{2}\eta(t_n)\int_{0}^{t_n}e^{(s-t)\partial_x^3}\eta(s)\Pi_{N_{\lambda}}\partial_{x}\left(\Pi_{N_{\lambda}}\F(s,\U^{\ell-1})\right)^2\d s.
	\end{align*}
	By subtracting the above equation from $\U^{\ell}(t)=\Gamma(\U^{\ell-1})(t)$ for $t\in[t_n,t_{n+1}]$, we derive
	\begin{align}\label{eq:estimate_um}
		\mathcal{N}_n^\ell(t)
		& =\left[\eta(t)-\eta(t_n)\right]e^{-t \partial_x^3}\Pi_{N_{\lambda}}\phi_{\lambda}\notag\\
		&\quad+\frac{1}{2}\left[\eta(t)-\eta(t_n)\right]\int_{0}^{t_n}e^{(s-t)\partial_x^3}\eta(s)\Pi_{N_{\lambda}}\partial_{x}\left(\Pi_{N_{\lambda}}\U^{\ell-1}(s)\right)^2\d s\notag\\
		&\quad-\left[\eta(t)-\eta(t_n)\right]\int_{0}^{t_n}e^{(s-t)\partial_x^3}\eta(s)\Pi_{N_{\lambda}}\partial_{x}\Big(\Pi_{N_{\lambda}}\F(s,\U^{\ell-1})\Pi_{N_{\lambda}}\U^{\ell-1}(s)\Big)\d s\notag\\
		&\quad+\frac{1}{2}\left[\eta(t)-\eta(t_n)\right]\int_{0}^{t_n}e^{(s-t)\partial_x^3}\eta(s)\Pi_{N_{\lambda}}\partial_{x}\left(\Pi_{N_{\lambda}}\F(s,\U^{\ell-1})\right)^2\d s\notag\\
		&\quad+\frac{1}{2}\eta(t)\int_{t_n}^{t}e^{(s-t)\partial_x^3}\eta(s)\Pi_{N_{\lambda}}\partial_{x}\left(\Pi_{N_{\lambda}}\U^{\ell-1}(s)\right)^2\d s\notag\\
		&\quad-\eta(t)\int_{t_n}^{t}e^{(s-t)\partial_x^3}\eta(s)\Pi_{N_{\lambda}}\partial_{x}\Big(\Pi_{N_{\lambda}}\F(s,\U^{\ell-1})\Pi_{N_{\lambda}}\U^{\ell-1}(s)\Big)\d s\notag\\
		&\quad+\frac{1}{2}\eta(t)\int_{t_n}^{t}e^{(s-t)\partial_x^3}\eta(s)\Pi_{N_{\lambda}}\partial_{x}\left(\Pi_{N_{\lambda}}\F(s,\U^{\ell-1})\right)^2\d s.
	\end{align}
	Noting that $|\eta(t)-\eta(t_n)|\lesssim (|\eta(t)|+|\eta(t_n)|)^{\frac{2}{3}}|\eta(t)-\eta(t_n)|^{\frac{1}{3}}\lesssim \tau^{\frac{1}{3}}_{\lambda}$, we obtain the estimate
	\begin{align*}
		\big\|\left[\eta(t)-\eta(t_n)\right]e^{-t \partial_x^3}\Pi_{N_{\lambda}}\phi_{\lambda}\big\|_{L^2_x}\lesssim \tau^{\frac{1}{3}}_{\lambda}N_{\lambda}^{-s}\|\phi_{\lambda}\|_{H^s}.
	\end{align*}
	For the remaining terms on the right-hand side of \eqref{eq:estimate_um}, since $t,t_n\lesssim 1$, whenever $\eta(t),\eta(t_n)\neq 0$, we use the estimate $|\eta(t)-\eta(t_n)|\lesssim\tau_{\lambda}$ and  apply the inequality \eqref{eq:rough_estimate} repeatedly. This gives
	\begin{align}\label{eq:u-eu}
		\left\|\mathcal{N}_n^{\ell}(t)\right\|_{L^2_x}
		&\lesssim \tau^{\frac{1}{3}}_{\lambda}N_{\lambda}^{-s}\|\phi_{\lambda}\|_{H^s}+\tau_{\lambda}N_{\lambda}^{\frac{3}{2}}\|\Pi_{N_{\lambda}}\U^{\ell-1}\|^2_{L^\infty_t L^2_x}+\tau_{\lambda}N_{\lambda}^{\frac{3}{2}}\|\Pi_{N_{\lambda}}\F(t,\U^{\ell-1})\|^2_{L^\infty_t L^2_x}\notag\\
		&\lesssim \tau^{\frac{1}{3}}_{\lambda}N_{\lambda}^{-s}\|\phi_{\lambda}\|_{H^s}+\tau_{\lambda}N_{\lambda}^{\frac{3}{2}}\|\Pi_{N_{\lambda}}\U^{\ell-1}\|^2_{L^\infty_t L^2_x}+\Big(\tau_{\lambda}N_{\lambda}^{\frac{3}{2}}\Big)^3\|\Pi_{N_{\lambda}}\U^{\ell-1}\|^4_{L^\infty_t L^2_x},
	\end{align}
	for all $t\in[t_n,t_{n+1}]$, where we have used \eqref{eq:rough_estimate} for $\|\Pi_{N_{\lambda}}\F(t,\U^{\ell-1})\|_{L^\infty_t L^2_x}$ in the last inequality. Substituting \eqref{eq:u-eu} into \eqref{eq:estimate_I2} and applying Lemma~\ref{lem:embedding_Ys} and Lemma~\ref{lem:Bernstein}, 
	we derive
	\begin{align}\label{eq:I_2}
		I_2&\lesssim  \tau_{\lambda}^{\frac{8}{3}} N_{\lambda}^{3-2s}\|\phi_{\lambda}\|_{H^s}^2\|\Pi_{N_{\lambda}}\U^{\ell-1}\|^2_{L^\infty_t L^2_x}+\left(\tau_{\lambda}^2 N_{\lambda}^{3}\right)^2 \|\Pi_{N_{\lambda}}\U^{\ell-1}\|^6_{L^\infty_t L^2_x}\notag\\
		&\quad+\left(\tau_{\lambda}^2 N_{\lambda}^{3}\right)^4 \|\Pi_{N_{\lambda}}\U^{\ell-1}\|^{10}_{L^\infty_t L^2_x}\notag\\
		&\lesssim \tau_{\lambda}^{\frac{8}{3}} N_{\lambda}^{3-4s}\|\phi_{\lambda}\|_{H^s}^2\|\U^{\ell-1}\|^2_{Y^s}+\left(\tau_{\lambda}^2 N_{\lambda}^{3}\right)^2 N_{\lambda}^{-6s}\|\U^{\ell-1}\|^6_{Y^s}+\left(\tau_{\lambda}^2 N_{\lambda}^{3}\right)^4 N_{\lambda}^{-10s} \|\U^{\ell-1}\|^{10}_{Y^s}.
	\end{align}
	Similarly, for $I_3$, we  use the inequalities \eqref{eq:rough_estimate} and \eqref{eq:u-eu} to obtain
	\begin{align}\label{eq:I_3}
		I_3&\lesssim \tau_{\lambda}^2 N_{\lambda}^{3}
		\big(\sup_{n}
		\left\|\mathcal{N}_n^\ell(t)\right\|^2_{L^{\infty}_t(t_n,t_{n+1};L^2(0,\lambda))}\big)^2\notag\\
		&\lesssim \tau_{\lambda}^{\frac{10}{3}} N_{\lambda}^{3-4s}\|\phi_{\lambda}\|_{H^s}^4+\left(\tau_{\lambda}^2 N_{\lambda}^{3}\right)^3 N_{\lambda}^{-8s}\|\U^{\ell-1}\|^8_{Y^s}+\left(\tau_{\lambda}^2 N_{\lambda}^{3}\right)^7 N_{\lambda}^{-16s}\|\U^{\ell-1}\|^{16}_{Y^s}.
	\end{align}
	Now, substituting \eqref{eq:I_1}, \eqref{eq:I_2} and \eqref{eq:I_3} into \eqref{eq:F_L2L2}, 
	we derive the following estimate:
	\begin{align}\label{eq:F_L2L2_new}
		\left\|\eta(t)\Pi_{N_{\lambda}}\F(t,\U^\ell)\right\|_{L^2_t L^2_x}
		& \lesssim \tau_{\lambda}N_{\lambda}^{1-2s}\left\|\U^{\ell}\right\|^2_{X_{s,\frac{1}{2}}}+\tau_{\lambda}^\frac{4}{3} N_{\lambda}^{\frac{3}{2}-2s}\|\phi_{\lambda}\|_{H^s}\|\U^{\ell-1}\|_{Y^s}\notag\\
		&\quad+\tau_{\lambda}^{\frac{5}{3}} N_{\lambda}^{\frac{3}{2}-2s}\|\phi_{\lambda}\|_{H^s}^2+\sum_{j=2}^{7}\tau_{\lambda}^{j}N_{\lambda}^{[\frac{3j}{2}-(j+1)s]}\|\U^{\ell-1}\|^{j+1}_{Y^s}.
	\end{align}
	Combining \eqref{eq:A_1} with the inequality $\|\U^{\ell}\|_{X_{s,\frac{1}{2}}}\leq \|\U^\ell\|_{Y^s}$ leads to the estimate for $ \mathcal{A}_1(\U^\ell)$:
	\begin{align}\label{eq:A_1_Ys}
		\|\mathcal{A}_1(\U^\ell)\|_{Y^s}
		&\lesssim \tau_{\lambda}N_{\lambda}^{2-2s}\left\|\U^{\ell}\right\|^3_{Y^s}+\tau_{\lambda}^\frac{4}{3} N_{\lambda}^{\frac{5}{2}-2s}\|\phi_{\lambda}\|_{H^s}\|\U^{\ell-1}\|_{Y^s}\left\|\U^{\ell}\right\|_{Y^s}\notag\\
		&\quad+\tau_{\lambda}^\frac{5}{3} N_{\lambda}^{\frac{5}{2}-2s}\|\phi_{\lambda}\|_{H^s}^2\left\|\U^{\ell}\right\|_{Y^s}+\sum_{j=2}^{7}\tau_{\lambda}^{j}N_{\lambda}^{[1+\frac{3j}{2}-(j+1)s]}\|\U^{\ell-1}\|^{j+1}_{Y^s}\left\|\U^{\ell}\right\|_{Y^s}\notag\\
		& =:\mathcal{C}\left(s,\tau_{\lambda},N_{\lambda},\left\|\U^{\ell}\right\|_{Y^s},\|\U^{\ell-1}\|_{Y^s}\right).
	\end{align}
	
	{\sc Estimate of $\|\mathcal{A}_2(\U^\ell)\|_{Y^s}$:} We estimate $\|\mathcal{A}_2(\U^\ell)\|_{Y^s}$ using the same approach as in \eqref{eq:estimate_A_1}. Specifically, we have
	\begin{align}\label{eq:A_2}
		\|\mathcal{A}_2(\U^\ell)\|_{Y^s}
		&\lesssim N_{\lambda}^{1+s}
		\big\|\eta^{\frac{1}{2}}(t)\Pi_{N_{\lambda}}\F(t,\U^\ell)\big\|_{L^2_t L^2_x}
		\big\|\eta^{\frac{1}{2}}(t)\Pi_{N_{\lambda}}\F(t,\U^\ell)\big\|_{L^4_t L^4_x}.
	\end{align}
	Using the scaling invariance of the Sobolev inequality and the fact that the support of $\eta$ is contained within the interval $[-2,2]$, we obtain the following estimate:
	\begin{align*}
		\left\|\eta^{\frac{1}{2}}(t)\Pi_{N_{\lambda}}\F(t,\U^\ell)\right\|_{L^4_t L^4_x}\lesssim \left\|\Pi_{N_{\lambda}}\F(t,\U^\ell)\right\|_{L^\infty_t L^4_x}\lesssim \left\|\Pi_{N_{\lambda}}\F(t,\U^\ell)\right\|_{L^\infty_t H^{\frac{1}{4}}_x}.
	\end{align*}
	Thus, applying Bernstein's inequality and using the estimate \eqref{eq:rough_estimate}, we obtain
	\begin{align}\label{eq:F_L4L4}
		\left\|\eta^{\frac{1}{2}}(t)\Pi_{N_{\lambda}}\F(t,\U^\ell)\right\|_{L^4_t L^4_x}\lesssim N_{\lambda}^{\frac{1}{4}}\tau_{\lambda}N_{\lambda}^{\frac{3}{2}}\left\|\Pi_{N_{\lambda}} \U^{\ell}\right\|^2_{L_{t}^{\infty}L_{x}^2}\lesssim \tau_{\lambda}N_{\lambda}^{\frac{7}{4}-2s}\left\| \U^{\ell}\right\|^2_{Y^s}.
	\end{align}
	Combining this result with \eqref{eq:F_L2L2_new} and \eqref{eq:A_2} yields the following estimate
	\begin{align}\label{eq:A_2_Ys}
		&\|\mathcal{A}_2(\U^\ell)\|_{Y^s}\lesssim \tau_{\lambda}N_{\lambda}^{\frac{7}{4}-s}\left\|\U^{\ell}\right\|_{Y^s}\cdot \mathcal{C}\left(s,\tau_{\lambda},N_{\lambda},\left\|\U^{\ell}\right\|_{Y^s},\|\U^{\ell-1}\|_{Y^s}\right),
	\end{align}
	where $\mathcal{C}\left(s,\tau_{\lambda},N_{\lambda},\left\|\U^{\ell}\right\|_{Y^s},\|\U^{\ell-1}\|_{Y^s}\right)$ is a function defined by \eqref{eq:A_1_Ys}.
	
	By substituting \eqref{eq:A_1_Ys} and \eqref{eq:A_2_Ys} into \eqref{eq:u_mp1}, we get the estimate for $\|\U^{\ell+1}\|_{Y^s}$ as follows
	\begin{align}\label{eq:u_mp1_Ys}
		\|\U^{\ell+1}\|_{Y^s}\leq& C_0\|\phi_{\lambda}\|_{H^s}+C\lambda^{0+}\left\|\U^\ell\right\|^2_{X_{s,\frac{1}{2}}}\notag\\
		&+C\Big(1+\tau_{\lambda}N_{\lambda}^{\frac{7}{4}-s}\left\|\U^{\ell}\right\|_{Y^s}\Big)\mathcal{C}\left(s,\tau_{\lambda},N_{\lambda},\left\|\U^{\ell}\right\|_{Y^s},\|\U^{\ell-1}\|_{Y^s}\right).
	\end{align}
	Next, we apply the inductive hypothesis $\|\U^\ell\|_{Y^s},\|\U^{\ell-1}\|_{Y^s}\leq C^{\#}\|\phi_{\lambda}\|_{H^s}$ along with relations $\|\phi_{\lambda}\|_{H^s}\leq\lambda^{-\frac{3}{2}-s}\|\phi\|_{H^s}$, $\tau_{\lambda}=\lambda^3\tau$ and $N_{\lambda}=\lambda^{-1}N$ to obtain the following estimates:
	\begin{align}
		\lambda^{0+}\left\|\U^\ell\right\|^2_{X_{s,\frac{1}{2}}}&\leq \left[(C^{\#})^2\lambda^{-\frac{3}{2}-s+}\|\phi\|_{H^s}\right]\cdot\|\phi_{\lambda}\|_{H^s},\label{eq:bound_constant_1}\\
		\tau_{\lambda}N_{\lambda}^{\frac{7}{4}-s}\left\|\U^{\ell}\right\|_{Y^s}&\leq \left[C^{\#}\lambda^{-\frac{1}{4}}\tau N^{\frac{7}{4}-s}\right]\cdot\|\phi\|_{H^s},\label{eq:bound_constant_2}
	\end{align}
	and
	\begin{align}\label{eq:bound_constant_3}
		&\mathcal{C}\left(s,\tau_{\lambda},N_{\lambda},\left\|\U^{\ell}\right\|_{Y^s},\|\U^{\ell-1}\|_{Y^s}\right)\notag\\
		&\leq (C^{\#})^3 \|\phi_{\lambda}\|_{H^s} 
		\Big[\lambda^{-2}\tau N^{2-2s}\|\phi\|_{H^s}^2 
		+ \lambda^{-\frac{3}{2}}\tau^{\frac{4}{3}} N^{\frac{5}{2}-2s}\|\phi\|^2_{H^s}\notag\\
		&\quad+\lambda^{-\frac{1}{2}}\tau^{\frac{5}{3}} N^{\frac{5}{2}-2s}\|\phi\|^2_{H^s}
		+\sum_{j=2}^{7} \lambda^{-\frac{5}{2}}\tau^{j}N^{[1+\frac{3j}{2}-(j+1)s]}\|\phi\|^{j+1}_{H^s}\Big].
	\end{align} 
	It is easy to observe that the coefficients of term $\|\phi_{\lambda}\|_{H^s}$ in the inequalities \eqref{eq:bound_constant_1}--\eqref{eq:bound_constant_3} all contain negative powers of the scaling constant $\lambda$. Moreover, under condition $N\lesssim \tau^{-\frac{1}{2-2s}}$ for $s\in(-\frac{1}{2},0]$, the following boundedness property holds:
	\begin{align*}
		\max \left\{\tau N^{\frac{7}{4}-s},\; \tau N^{2-2s},\; \tau^{\frac{4}{3}} N^{\frac{5}{2}-2s}, \;\tau^{\frac{5}{3}} N^{\frac{5}{2}-2s}\|\phi\|^2_{H^s},\; \max_{j\in\{2,\cdots,7\}}\tau^{j}N^{[1+\frac{3j}{2}-(j+1)s]}\right\}\lesssim 1.
	\end{align*}
	Thus, by combining \eqref{eq:u_mp1_Ys},  there exists a constant 
	$\lambda_0>1$ independent of $\tau$ and $N$ but possibly dependent on $\|\phi\|_{H^s}$ such that for $\lambda\geq \lambda_0$,   the following estimate holds:
	\begin{align}
		\|\U^{\ell+1}\|_{Y^s}\leq2C_0 \|\phi_{\lambda}\|_{H^s}=C^{\#}\|\phi_{\lambda}\|_{H^s}.
	\end{align}
	Thus, by induction, we have proven the conclusion of Lemma~\ref{lem:boundedness_um}.
\end{proof}

\end{document}